\documentclass{amsart}
\usepackage[utf8]{inputenc}
\usepackage{amsmath,amssymb,mathabx,comment}
\usepackage{hyperref}
\usepackage{amsthm}
\newtheorem{thm}{Theorem}
\newtheorem{definition}{Definition}
\newtheorem{lemma}{Lemma}[section]
\newtheorem{prop}{Proposition}[section]
\newtheorem{rmk}{Remark}[section]

\newcommand{\cP}{{\mathcal{P}}}
\newcommand{\cU}{{\mathcal{U}}}

\newcommand{\E}{{\mathbf{E}}}

\newcommand{\cf}{{\mathcal{F}}}
\newcommand{\V}{V_*}

\newcommand{\f}{\dot f_*}

\usepackage{tikz,floatrow}

\title[Exponential Fermi Acceleration]{Exponential Fermi Acceleration in a Switching Billiard}

\date{\today}

\begin{document}

\author{Davit Karagulyan$^*$, Jing Zhou$^\dagger$}
\thanks{$^*$ Department of Mathematics, Royal Institute of Technology, Sweden. Email: davitk@kth.se.\\ $^\dagger$ Department of Mathematics, Penn State University, USA. Email: jingzhou@psu.edu\\
The authors gratefully acknowledge the enlightening discussion with Dmitry Dolgopyat, which has greatly improved the structure of the proof in the paper.}

\maketitle

\begin{abstract}
    In this paper we show an infinite measure set of exponentially escaping orbits for a resonant Fermi accelerator, which is realised as a square billiard with a periodically oscillating platform. We use normal forms to describe how the energy changes in a period and we employ techniques for hyperbolic systems with singularities to show the exponential drift of these normal forms on a divided time-energy phase.
\end{abstract}

\tableofcontents

%-------------------------------------------------------------
\section{Introduction}
The Fermi acceleration model has been extensively investigated numerically and theoretically by many physicists and mathematicians in the past decades (c.f. \cite{Dol08} \cite{GRKT12} \cite{LLC80} for a survey on the subject). In 1949 Fermi \cite{Fermi49} proposed a model where charged particles travel in moving magnetic field, as an attempt to explain for the existence of high-energy particles in outer universe. Later in 1961 Ulam \cite{Ulam61} extracted the effective dynamics in this situation, i.e. a ball bounces between a fixed wall and a periodically oscillating wall, and he conjectured the existence of \emph{escaping orbits} (those whose energy grow to infinity in time) based on his numerical experiment with piecewise linear models. Since then, numerous efforts have been made towards locating escaping orbits as well as \emph{bounded} (those whose energy remain bounded) and \emph{oscillatory} ones (those whose energy have finite $\liminf$ but infinite $\limsup$) in the Fermi acceleration models and their variations.

For sufficiently smooth wall motions, KAM theory has implied that the prevalence of invariant curves forces the energy of every orbit to be bounded \cite{Pu83} \cite{LaLe91} \cite{Pu94}. Zharnitsky \cite{Zhar2000} generalized the situation to KAM quasi-periodic motions with Diophantine frequencies. Kunze and Ortega \cite{kuOr20} showed that the set of escaping orbits has zero measure for wall motions with rationally independent frequencies. 

The Fermi acceleration exhibits richer dynamical phenomena if we allow singularities in the wall motions. Zharnitsky \cite{Zhar98} found linearly escaping orbits in Ulam's piecewise linear models. De Simoi and Dolgopyat \cite{deSD12} showed in a piecewise smooth model with one singularity that the system changes from elliptic behavior to hyperbolic behavior as a critical system parameter varies: in the elliptic regime the set of escaping orbits has infinite measure while in the hyperbolic regime it has zero measure but full Hausdorff dimension. 

When the fixed wall is replaced by background potential, escaping orbits are shown to have infinite measure for periodic analytic wall motions with gravity \cite{Pu77}. However, if we allow one singularity in the gravity model and impose hyperbolicity assumptions, the escaping orbits constitute a null set and they co-exist with bounded orbits at arbitrarily high energy levels \cite{Zhou21}. Non-constant background potential models have also been examined and we refer to \cite{DS09} \cite{Dol08potential} \cite{Or02} for examples.

The billiard model is a natural realization of the Fermi acceleration model in two dimensions. For instance, unbounded orbits have been located in a billiard model with smoothly breathing boundaries \cite{GT08ham} \cite{KMKC95} \cite{LKSK04}. Exponentially escaping orbits have been constructed in a non-autonomous billiard model \cite{GT08nonauto} and are conjectured to be generic in an oscillating mushroom model \cite{GRKT14}. However in general it remains challenging to construct a Fermi accelerator with a full measure or even positive measure set of exponentially accelerating orbits. 

In this paper, we study a unit square billiard with a vertically oscillatory slit of infinite mass on the left. A ball bounces elastically on the boundaries of the square and the slit. The slit has length $\lambda$ and its motion (height) is described by $f(t)$ which is periodic in two $f(t+2)=f(t)$ and piecewise $C^3$ with a jump discontinuity of $\dot f(t)$ at $t=2n$ ($n\in\mathbb{Z}$) (c.f. Figure \ref{fig:slit}). The horizontal speed of the ball is 1 so the horizontal motion of the ball is in 1:1 resonance to the movement of the slit.

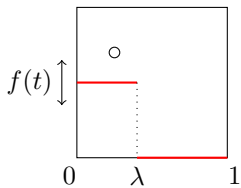
\begin{figure}[h]
    \label{fig:slit}
    \centering
       \begin{tikzpicture}
          \draw (0,0) rectangle (2,2);
          \draw[red,thick] (0.8,1) -- (0,1)
                           (0.8,0) -- (2,0);
          \draw[<->] (-0.2,0.7) -- (-0.2,1.3);
          \draw (0.5,1.4) circle (2pt);
          \draw (-0.2,1) node[anchor=east]{$f(t)$}
                (-0.1,0) node[anchor=north]{0}
                (2.1,0) node[anchor=north]{1};
          \draw[dotted] (0.8,1) -- (0.8,0) node[anchor=north]{$\lambda$};
       \end{tikzpicture}
    \caption{A Switching Billiard}
\end{figure}

We acquire this model from \cite{STR} where Shah, Turaev and Rom-Kedar proposed an ideally probabilistic approximation that the probability of jumping up or down is proportional to the lengths of openings which well matched the numerically observed exponential growth rate in the non-resonant case and the escape rate in the resonant case was numerically observed to be significantly higher than that in the non-resonant case.

This paper addresses an open question from \cite{Zhou20} where the non-switching case with 1:1 resonance was discussed. In \cite{Zhou20} when the relative positions of two oscillating slits 
are different at two critical jumps in a period and a trapping region, either the upper or the lower chamber of the table, is therefore created so that the ball can never escape to the other chamber once it enters the trapping chamber and in this situation almost every high energy orbit eventually gains energy exponentially fast. However, in the current setting when there is only one slit, trapping regions do not exist and hence the old mechanism of exponential acceleration no long works.

Let $x_0$ denote the starting horizontal position of the ball. We assume without loss of generality that the ball starts from the left part of the table, i.e. $0< x_0<\lambda$. We name as the \emph{interacting region} the slit on the left and the floor on the right (c.f. the red region in Figure \ref{fig:slit}). We record the time $t$ and the vertical velocity $v$ of the ball immediately after the collision at the interacting region. Let $F$ denote the collision map which sends a collision $(t,v)$ to the next $(\bar{t},\bar{v})$. We exclude from our discussion the collisions at the edge of the slit, which constitute a null set from the $(t,v)$ phase space. It takes time $\Delta t=2$ for the ball to finish a complete revolution. We denote by $\cf$ the dynamics of a complete revolution and $P$ a half revolution, whose definitions are made precise in Section 3. We are interested in the size of the escaping orbits in this billiard system and we show that for a large collections of slit motions $f$ there exists an infinite measure set of exponentially escaping orbits.

%The paper is structured as follows. First we state the main result in Section 2. Secondly in Section 3 we derive the normal forms for the half-revolution map $P$ and discuss the structure of the phase cylinders on which our dynamics sits. The proof of the main result is dissected into 4 sections: In Section 4 we show that under our assumptions the dynamical system is hyperbolic with an invariant unstable cone; in Section 5 we prove all the auxilliary lemmas such as the distortion control and the growth lemmas as an outcome of the nice hyperbolic behaviors of our system; then in Section 6 we show the system possesses exponential drift in energy in finite time-intervals; and finally in Section 7 we prove the eventual exponential acceleration as time goes to infinity.

%-------------------------------------------------------------
\section{Main Result}
We recall that $x_0$ is the starting horizontal position of the ball and $\lambda$ is the length of the slit. We may assume without loss of generality that we start from the left chamber, i.e. $0<x_0<\lambda$. Since the horizontal motion of the ball is in 1:1 resonance with the vertical oscillation of the slit, in a period $\Delta t=2$ the ball switches from the left chamber to the right at the moment $t_1^*=\lambda-x_0$ and later from right to left at $t_2^*=2-\lambda-x_0$. We denote by $f_i:=f(t_i^*)$ and $\dot f_i :=\dot f(t_i^*)$ ($i=1,2$).

\begin{thm}[Main Theorem]\label{expAcc}
Assume that $f(t)$ is periodic in two $f(t+2)=f(t)$ and piecewise $C^3$ with a jump discontinuity of $\dot f(t)$ at $t=2n$ ($n\in\mathbb{Z}$), and that $f_1 \neq f_2$. Let
$$
\mathcal E=(1-f_2)\log \frac{1-f_2}{1-f_1}+f_2 \log \frac{f_2}{f_1}.
$$
Assume also that for all $t$, $c \le f(t) \le 1-c$ for some constant $c\in(0,1/2)$. Then there exists $\f>0$ large so that if $|\dot f_1|,|\dot f_2|>\f$ then one can choose an energy threshold  $V_*=V_*(f_1, f_2,\dot f_1,\dot f_2)$ in such a way that there exist positive constants $\alpha$, $\beta, T_*$ and $\iota> 1$, depending on $f_1, f_2, \mathcal E$, so that for every $V_1,V_2$ and $T\ge T_*$ with $V_* \iota^T<V_1 \ll V_2$, we have
$$
\frac{|\{v_0 \in [V_1,V_2]: v_t \geq e^{\alpha t}v_0, \forall t \geq T \}|}{|V_2-V_1|}\geq 1-e^{-\beta T}. 
$$
\end{thm}

Without loss of generality we can assume that $f_2>f_1$, so that the ball accelerates when it enters the lower chamber and decelerates when it enters the upper chamber (c.f. Section 3.2). The case when $f_2<f_1$ can be proven analogously.\\

The main tool we use to prove the exponential acceleration is the normal forms derived in Section 3, which describes how the energy of the ball will change in a period when it enters the upper chamber or the lower chamber. The time-energy phase for these normal forms is divided into the accelerating lower route and the decelerating upper route (c.f. Figure \ref{fig:dividedphase}) under the assumption $f_2>f_1$. Our strategy of proof proceeds as follow: we first show that under the assumptions in Theorem \ref{expAcc} our system enjoys strong uniform hyperbolicity so that ideally if we start with an expanding curve it would get stretched considerably and we can estimate the accelerating and decelerating portions in its image and repeat the process. However, since our phase plane is divided, a long expanding curve will be cut into many pieces some of which might remain tiny for a long time so that its images after long time will be highly fragmented and we cannot carry out a uniform analysis. We fix the problems by proving some growth lemmas which control the size of tiny curves and revive the analysis. The other technical difficulty is that the normal forms only hold above certain energy threshold $V_*$ and the analysis collapses when the energy of ball falls below it. We take care of this by first proving the result for a modified system whose energy never falls below the threshold and then we argue that a considerable portion of accelerating orbits still survive for the original system.\\

More precisely, the proof of the Main Theorem \ref{expAcc} is structured as follow. We first derive the normal forms in Section 3 which captures how the energy of the ball changes after one period. Then in Section 4 we show that our system is uniformly hyperbolic under the assumptions in Theorem \ref{expAcc} in the sense that the normal forms share an invariant unstable cone. In Section 5 we modify the definition of our system by pushing up the energy of the ball whenever it falls near the critical energy threshold $V_*$. Next in Section 6 we estimate the size and the growth rate of unstable curves. Then we prove for the modified system in Section 7 that the assumption $f_2 \neq f_1$ yields exponential acceleration in finite time for sufficiently large expanding curves (called ``long curves"). Finally in Section 8 we prove exponential acceleration for infinite time by using some large deviation and moment estimate arguments combined with the quantitative version of the growth lemma proven in Section 6.

%-------------------------------------------------------------
\section{Preliminaries}
In this section we discuss the structure of our dynamics: the map $\cf$ of a complete revolution of the ball and the divided phase cylinder it sits on.

We denote the two singularity curves in the $(t,v)$-phase plane as $S_i=\{t=t_i^*\}$; these are the collisions occurring at the edge of the slit and the ball motion is not well-defined.

For $i=1,2$, let $R_i$ denote the singular strip bounded by $S_i$ and its image $f(S_i)$ respectively. We note that $R_1$ collects the collisions at the right floor immediately after the ball exits the left chamber and $R_2$ the first collisions at the slit after the ball enters the left chamber. We further subdivide $R_2$ into $R_2^+$ and $R_2^-$ for left upper and left lower chambers. We also denote the preimages as $\Tilde{R}_i=f^{-1}(R_i)$. In the following text we omit the subscript $i$ whenever it is clear from the context.

The behavior of the ball differs when it stays in one chamber from when it switches chambers. We first describe in Section \ref{adiabtic} how the energy of ball (not) changes when it stays in one chamber and we then describe in Section \ref{normal} how it changes when the ball jumps from one chamber to the other. There are two possible routes of the ball when it switches chambers in a period: the ball enters the upper left chamber and then return to the floor (\emph{the upper route}); or it enters the lower left chamber and then return to the floor (\emph{the lower route}). Our goal is to describe the first return map $\cf$ to the right floor $R_1$ for the two routes. The adiabatic coordinates and the normal forms in Section 4 from \cite{Zhou20} readily apply in our case if we think of the right slit in \cite{Zhou20} to be always on the floor: there exists the critical initial energy $V_*$ sufficiently large that we have the adiabatic coordinates and the normal forms in Section \ref{adiabtic} and Section \ref{normal} respectively.

%-------------------------------------------------------------
\subsection{The Adiabatic Coordinates}\label{adiabtic}
In this section we describe the behavior of the ball when it stays in the left or right chamber. We tailor the adiabatic coordinates from \cite{Zhou20} for our case.

First we assume that the ball hits the slit from above and it does not enter or exit the left chamber in the next collision. We introduce $l(t)$ as the ``distance'' from the slit to the ceiling: 
\[ l(t) = 
   \begin{cases}
      1-f(t) & 0 \le t \le t_1^*, t_2^* \le t \le 2 \\
      1      & \hbox{otherwise}
   \end{cases}.
\]
We need the following normalising constant $\displaystyle \mathcal{L}_*= \int_0^2 \frac{ds}{l(s)^2}$. We introduce the notation $\psi=\mathcal{O}_s(v^{-n})$ as $v^n\psi$ is bounded and so are its derivatives up to $s$-order.

\begin{lemma}\label{uppercoord}
For $(t,v)\in\{0\le t < t_1^*, t_2^*<t \le 2,\}\backslash R_1 \cup R_2 \cup \tilde{R}_1 \cup \tilde{R}_2$ and $v\gg 1$, there exists an adiabatic coordinate $(\theta,I) = \Psi_U (t,v) \in \mathbb{R}/2\mathbb{Z} \times \mathbb{R}_+$ such that 
\[
\theta_{n+1} = \theta_n + \frac{2}{I_n} + \mathcal{O}_3\left( \frac{1}{I_n^4} \right), \ I_{n+1} = I_n + \mathcal{O}_3\left(\frac{1}{I_n^3}\right).
\]
In fact, $\displaystyle \theta =\theta(t) =\frac{2}{\mathcal{L}_*} \int_0^t \frac{ds}{l(s)^2} \mod 2$, $\displaystyle I = I(t,v) = \frac{\mathcal{L}_*}{2}\left(lv +l\dot{l} + \frac{l^2\ddot{l}}{3v}\right)$.
\end{lemma}

Similarly, we have an adiabatic coordinate when the ball hits the slit from below. We suppose that the ball does not enter or exit the left chamber at the next collision. We introduce $m(t)$ as the ``distance'' from the slit to the floor of the table
\[ m(t) = 
\begin{cases}
   - f(t) & 0 \le t \le t_1^*, t_2^* \le t \le 2 \\
     -1   & \hbox{otherwise}
\end{cases}.
\]
We need the following normalising constant $\displaystyle \mathcal{M}_*= \int_0^2 m(s)^{-2} ds$.

\begin{lemma}\label{lowercoord}
For $(t,v)\in\{0\le t < t_1^*, t_2^*<t \le 2\}\backslash R_1 \cup R_2 \cup \tilde{R}_1 \cup \tilde{R}_2$ and $v\ll -1$, there exists an adiabatic coordinate $(\zeta,J) = \Psi_L (t,v) \in \mathbb{R}/2\mathbb{Z} \times \mathbb{R}_+$ such that 
\[
\zeta_{n+1} = \zeta_n + \frac{2}{J_n} + \mathcal{O}_3\left( \frac{1}{J_n^4} \right), \ J_{n+1} = J_n + \mathcal{O}_3\left(\frac{1}{J_n^3}\right).
\]
In fact, $\displaystyle \zeta =\zeta(t) =\frac{2}{\mathcal{M}_*} \int_0^t \frac{ds}{m(s)^2} \mod 2$, $\displaystyle J = J(t,v) = \frac{\mathcal{M}_*}{2}\left(mv +m\dot{m} + \frac{m^2\ddot{m}}{3v}\right)$.
\end{lemma}

Finally, it might appear awkward but let us present also the adiabatic coordinate when the ball interact with the floor on the right; not only for the sake of completeness do we write this but it will also make clear the presentation in the future sections.

\begin{lemma}\label{floorcoord}
For $(t,v)\in\{t_1^*<t<t_2^*\}\backslash R_1 \cup R_2 \cup \tilde{R}_1 \cup \tilde{R}_2$ and $v\gg 1$, there exists an adiabatic coordinate $(\theta,H) = \Psi_F (t,v) \in \mathbb{R}/2\mathbb{Z} \times \mathbb{R}_+$ such that 
\[
\theta_{n+1} = \theta_n + \frac{2}{I_n}, \ H_{n+1} = H_n.
\]
In fact, $\displaystyle \theta =\theta(t) =\frac{2}{\mathcal{L}_*} \int_0^t \frac{ds}{l(s)^2} \mod 2$, $\displaystyle H = H(t,v) = \mathcal{L}_*v/2$.
\end{lemma}

%-------------------------------------------------------------

\subsection{The Normal Forms}\label{normal}
In this section we describe the behavior of the ball when it jumps from one chamber to the other. We tailor the normal forms from \cite{Zhou20} for our case and all the normal forms are valid for sufficiently large initial energies $V_*$.

We decompose the Poincar\'e map $P$ on $R_1$ (the first return map to the right floor) into two maps between the two singular strips $R_1, R_2$. We recall that the ball may follow the upper or lower route in a period.

\subsubsection{The Upper Route}
In the upper route, $P=P^U$ decomposes into the following two maps $P^{U}_{21}\circ P^{U}_{12}$: $P^{U}_{12}$ describes the behavior of the ball from $R_1$ (the right foor) to $R_2^+$ (the upper left chamber) whereas $P^{U}_{21}$ from $R_2^+$ back to $R_1$.

We introduce a new pair of variables $(\sigma, \mathcal{H})$ defined on $R_1$:
\[
\sigma = H(\theta - \theta_1^*),\  
\mathcal{H}=H/\mathcal{L}_*
\]
where $\displaystyle \theta_1^* = \frac{2}{\mathcal{L}_*}\int_0^{t_1^*} \frac{ds}{l(s)^2}$, and a new pair of variables $(\tau, \mathcal{I})$ on $R_2^+$
\[
\tau = I(\theta - \theta_2^*),\  \mathcal{I}=\frac{I}{\mathcal{L}_*}
\]
where $\displaystyle \theta_2^* = \frac{2}{\mathcal{L}_*}\int_0^{t_2^*} \frac{ds}{l(s)^2}$.

We need the following constants $(i=1,2)$:
\begin{align*}
&\Delta_i = \frac{1}{2} \frac{l(t_i^*+)}{l(t_i^*-)} \left( l(t_i^*-)\dot{l}(t_i^*+) - l(t_i^*+)\dot{l}(t_i^*-) \right),\\
&\Delta'_i = \frac{1}{8} l(t_i^*+)^2 \left( l(t_i^*-)\ddot{l}(t_i^*+) - l(t_i^*+)\ddot{l}(t_i^*-) \right),\\
&\Delta''_i = \frac{1}{24} l(t_i^*-)l(t_i^*+) \left( l(t_i^*-)\ddot{l}(t_i^*+) - l(t_i^*+)\ddot{l}(t_i^*-) \right).
\end{align*}
We write in brevity $f_1=f(t_1^*)$, etc. throughout the following texts.

\begin{prop}[Upper Route]\label{uppernormal}
Suppose that $(\sigma, \mathcal{H}) \in R_1$ and $\mathcal{H}>V_*$, and that 
$$f_2 \lesssim \{ \mathcal{L}_*\mathcal{H}(\theta_2^* - \theta_1^*) -\sigma \}_2 \lesssim 2-f_2,$$ 
where $\lesssim$ means the inequality holds up to an error of order $\mathcal{O}(\frac{1}{\mathcal{H}})$, 
and $\{\bullet\}_2 = \bullet \mod 2$. Then the Poincar\'e map $P^{U}_{12}: R_1 \to R_2^+$ is given by $(\tau,\mathcal{I}) = G^{U}_{12}(\sigma, \mathcal{H}) + H^{U}_{12}(\sigma, \mathcal{H}) + \mathcal{O}_3(\mathcal{H}^{-2})$ where 
\[ 
G^{U}_{12}(\sigma, \mathcal{H}) =\left( -\frac{1}{1-f_2} \{ \mathcal{L}_*\mathcal{H}(\theta_2^* - \theta_1^*) -\sigma \}_2 + \frac{2-f_2}{1-f_2}, (1-f_2) \mathcal{H} + \Delta_2 (\tau -1) \right)
\]
and
\[ 
H_{UU}^{12}(\sigma, \mathcal{H}) =\left(0, \Delta'_2 (\tau -1)^2/\mathcal{H} + \Delta''_2/\mathcal{H} \right).
\]
Similarly, suppose that $(\tau,\mathcal{I}) \in R_2^+$, $\mathcal{I} > V_*$. 
Then the Poincar\'e map $P^{U}_{21}: R_2^+ \to R_1$ is given by 
$$(\sigma,\mathcal{H}) = G^{U}_{21}(\tau,\mathcal{I}) + H^{U}_{21}(\tau,\mathcal{I}) + \mathcal{O}_3(\mathcal{I}^{-2})$$ 
where 
\[ 
G^{U}_{21}(\tau,\mathcal{I}) =\left( -(1-f_1) \{ \mathcal{L}_*\mathcal{I}(2+\theta_1^* - \theta_2^*) -\tau \}_2 + 2-f_1, \frac{1}{1-f_1} \mathcal{I} + \Delta_1 (\sigma -1) \right)
\] 
\[ 
H^{U}_{21}(\tau,\mathcal{I}) =\left(0, \Delta'_1 (\sigma -1)^2/\mathcal{I} + \Delta''_1/\mathcal{I} \right).
\]
\end{prop}

\subsubsection{The Lower Route}
In the lower route, depending on whether the ball travels the long trajectory or the short one when exiting or enter the lower left chamber (c.f. Figure \ref{fig:lowerroute}), $P=P^L$ decomposes into four possible combinations $P^{Ll}_{21}\circ P^{Ll}_{12}$, $P^{Ll}_{21}\circ P^{Ls}_{12}$, $P^{Ls}_{21}\circ P^{Ll}_{12}$ and $P^{Ls}_{21}\circ P^{Ls}_{12}$: $P^{Ll/Ls}_{12}$ describes the long/short route entering from $R_1$ (the right floor) to $R_2^-$ (the lower left chamber) whereas $P^{Ll/Ls}_{21}$ describes the long/short route exiting from $R_2^-$ back to $R_1$.

% Fig: The Lower Route
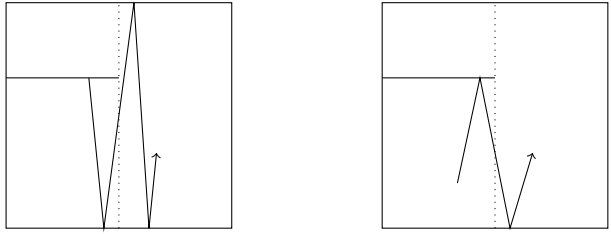
\begin{figure}[h!]
  \centering
     \begin{tikzpicture}
        \draw (0,0) rectangle (3,3)   (5,0) rectangle (8,3);
        \draw[dotted] (1.5,0)--(1.5,3) (6.5,0)--(6.5,3);
        \draw (0,2)--(1.5,2)   (5,2)--(6.5,2);
        \draw[->] (1.1,2)--(1.3,0)--(1.7,3)--(1.9,0)--(2,1);
        \draw[->] (6,0.6)--(6.3,2)--(6.7,0)--(7,1);
     \end{tikzpicture}
  \caption{Exiting the lower left chamber}\label{fig:lowerroute}
  \floatfoot{Long Exit on the left v.s. Short Exit on the right.}
\end{figure}

We introduce a new pair of variables $(\rho,\mathcal{J})$ on the lower singular strip $R_2^-$, which is the counterpart of $(\tau,\mathcal{I})$ on $R_2^+$ as follows 
\[ 
\rho = J(\zeta - \zeta_2^*), \ 
\mathcal{J}=J/\mathcal{M}_*
\]
where $\displaystyle \zeta_i^*=\frac{2}{M_*}\int_0^{t_i^*} \frac{ds}{m(s)^2}$.

First we describe the long/short entry from the right floor to the left lower chamber.

We need the following constants for the long entry:
\begin{align*}
   &\kappa_l= \frac{1}{2}m_+\dot{m}_+\\
   &\kappa_l'= \frac{1}{24}m_+^2\ddot{m}_+\\
   &\kappa_li''= \frac{1}{8}m_+^2\ddot{m}_+
\end{align*}
where all functions above take values at the moment $t=t_2^*$.

\begin{prop}[Long Entry]\label{longentry}
Assume that $(\sigma,\mathcal{H}) \in R_1^+$ with $\mathcal{H} > V_*$ and that 
$$\{ \mathcal{L}_*\mathcal{H}(\theta_2^* - \theta_1^*) -\sigma \}_2 \gtrsim 2-f_2.$$
Then the Poincar\'e map $P^{Ll}_{12}:R_1 \to R_2^-$ is given by 
$$(\rho,\mathcal{J})=G^{Ll}_{12}(\sigma,\mathcal{H}) + H^{Ll }_{12}(\sigma,\mathcal{H}) + \mathcal{O}_3(\mathcal{H}^{-2})$$ where
\[
G^{Ll}_{12}(\sigma,\mathcal{H})= \left(-\frac{1}{f_2}\{ \mathcal{L}_*\mathcal{H}(\theta_2^* - \theta_1^*) -\sigma \}_2 + \frac{f_2+2}{f_2}, f_2\mathcal{H} + \kappa_l(\rho-1)\right)
\]
and
\[
H^{Ll}_{12}(\sigma,\mathcal{H})=\left(0, \kappa_l'\frac{\rho-1}{H} -\kappa_l''\frac{(\rho-1)^3}{\mathcal{H}} \right).
\]
\end{prop}

We need the following constants to describe the short entry:
\begin{align*}
&\kappa_s'= \frac{1}{8}m_+^2\ddot{m}_+\\
&\kappa_s''= -\frac{1}{24}m_+m_+\ddot{m}_+
\end{align*}
where all functions above take values at the moment $t=t_2^*$.

\begin{prop}[Short Entry]\label{shortentry}
Assume that $(\sigma,\mathcal{H}) \in R_1$ with $\mathcal{H} > V_*$ and that 
$$f_2 \gtrsim \{ \mathcal{L}_*\mathcal{H}(\theta_2^* - \theta_1^*) -\sigma \}_2.$$ 
Then the Poincar\'e map $P^{Ls}_{12}:R_1 \to R_2^-$ is given by 
$$(\rho,\mathcal{J})=G^{Ls}_{12}(\sigma,\mathcal{H}) + H^{Ls}_{12}(\sigma,\mathcal{H}) + \mathcal{O}_3(\mathcal{H}^{-2})$$ where
\[
G^{Ls}_{12}(\sigma,\mathcal{H})=\left(-\frac{1}{f_2}\{ \mathcal{L}_*\mathcal{H}(\theta_2^* - \theta_1^*) -\sigma \}_2 + 1, f_2\mathcal{\mathcal{H}} + \kappa_l(\rho-1) \right)
\]
and
\[
H_{UL\rm II}^{12}(\sigma,\mathcal{H})=\left(0 , -\kappa_s'\frac{(\rho-1)^2}{\mathcal{H}} -\frac{\kappa_s''}{\mathcal{H}} \right).
\]
\end{prop}

Next we describe the long/short exit from the lower left chamber to the right floor.

We need the following constants for the long exit:
\begin{align*}
   &\chi_l= -\frac{1}{2} \dot{m}_-/m_-\\
   &\chi_l'= \frac{1}{8}\ddot{m}_-(1-\frac{1}{3}m_-^2)\\
   &\chi_l''= -\frac{1}{4}\ddot{m}_-
\end{align*}
where all functions above take values at the moment $t=t_1^*$.

\begin{prop}[Long Exit]\label{longexit}
Assume that $(\rho,\mathcal{J}) \in R_2^-$ with $\mathcal{J} > V_*$ and that 
$$\{ \mathcal{M}_*\mathcal{J}(2+\zeta_1^* - \zeta_2^*) -\rho \}_2 \gtrsim 1.$$ 
Then the Poincar\'e map $P^{Ll}_{21}: R_2^- \to R_1$ is given by 
$$(\sigma,\mathcal{H})=G^{Ll}_{21}(\rho,\mathcal{J}) + H^{Ll}_{21}(\rho,\mathcal{J}) + \mathcal{O}_3(\mathcal{H}^{-2})$$ 
where 
\[
G^{Ll}_{21}(\rho,\mathcal{J})=\left( -f_1\{ \mathcal{M}_*\mathcal{J}(2+\zeta_1^* - \zeta_2^*) -\rho \}_2 + 2 + f_1,
\frac{1}{f_1}\mathcal{J} + \chi_l(\sigma-2) \right)
\]
and
\[
H^{Ll}_{21}(\rho,\mathcal{J})=\left(0 , \frac{\chi_l'}{\mathcal{J}} +\chi_l''\frac{\sigma-1}{\mathcal{J}} -\frac{\chi_l''}{2}\frac{(\sigma-1)^2}{\mathcal{J}} \right).
\]
\end{prop}

We need the following constants to describe the short exit:
\begin{align*}
&\chi_s'= \frac{1}{4}\ddot{m}_-\\
&\chi_s''= \frac{1}{24}m_-(m_-^2\ddot{m}_- - 3\ddot{m}_-)
\end{align*}
where all functions above take values at the moment $t=t_1^*$.

\begin{prop}[Short Exit]\label{shortexit}
Assume that $(\rho,\mathcal{J}) \in R_2^-$ with $\mathcal{J} > V_*$ and that 
$$ \{ \mathcal{M}_*\mathcal{J}(2+\zeta_1^* - \zeta_2^*) -\rho \}_2 \lesssim 1.$$ 
Then the Poincar\'e map $P^{Ls}_{21}: R_2^- \to R_1$ is given by 
$$(\sigma,\mathcal{H})=G^{Ls}_{21}(\rho,\mathcal{J}) + H^{Ls}_{21}(\rho,\mathcal{J}) + \mathcal{O}_3(\mathcal{J}^{-2})$$ where
\[
G^{Ls}_{21}(\rho,\mathcal{J})=\left( -f_1 \{ \mathcal{M}_*\mathcal{J}(2+\zeta_1^* - \zeta_2^*) -\rho \}_2 + f_1,
\frac{1}{f_1}\mathcal{J} + \chi_l\sigma \right)
\]
and
\[
H^{Ls}_{21}(\rho,\mathcal{J})=\left(0 , \chi_s'\frac{\sigma-1}{\mathcal{J}} +\frac{\chi_s'}{2}\frac{(\sigma-1)^2}{\mathcal{J}} -\frac{\chi_s''}{\mathcal{J}} \right).
\]
\end{prop}

We note that the derivatives of the linear parts $G$ of $P^{Ll}_{12}$ and $P^{Ls}_{12}$ are identical, and so are those of $P^{Ll}_{21}$ and $P^{Ls}_{21}$.

%-------------------------------------------------------------

\subsection{The Divided Phase Cylinder}\label{divphase}
We dissect the behavior of the ball in a period into two stages: first it leaves the right floor and enters either the left upper or the left lower chamber, and then it exits the left chamber and returns to the right floor.

Depending on the ball's choice of entering the upper or lower left chamber, the $(\sigma,\mathcal{H})$-phase cylinder on the right floor $R_1$ are divided into three parts: the points in
$$
\mathcal{L}^l_{en}:= \{\{ \mathcal{L}_*\mathcal{H}(\theta_2^* - \theta_1^*) -\sigma \}_2 \gtrsim 2-f_2\}
$$ enter the lower left chamber following the Long Route, the points in 
$$
\mathcal{L}^s_{en}:=\{\{ \mathcal{L}_*\mathcal{H}(\theta_2^* - \theta_1^*) -\sigma \}_2 \lesssim f_2\}
$$ enter the lower left chamber following the Short Route, and the points in 
$$\mathcal{U}_{en}:=\{ f_2 \lesssim \{ \mathcal{L}_*\mathcal{H}(\theta_2^* - \theta_1^*) -\sigma \}_2 \lesssim 2-f_2\}$$ enter the upper left chamber (c.f. Figure \ref{fig:dividedphase} on the top). 

Note that $\mathcal{L}^l_{en}$, $\mathcal{L}^s_{en}$ and $\cU_{en}$ consist of connected components indexed by integers $m \geq 1$: for instance, a connected components $\cU_{en,m}$ of $\cU_{en}$ is defined by (see Figure \ref{fig:dividedphase})
$$
f_2 \lesssim f_2 \mathcal{L}_*\mathcal{H}(\theta_2^* - \theta_1^*) -\sigma -m \lesssim 2-f_2.
$$

We collect all the half-revolutions and define $$\mathcal{P}=\{P_{12}^{U}, P_{21}^{U}, P_{12}^{L l}, P_{12}^{Ls}, P_{21}^{Ls}, P_{21}^{Ll}\}.$$

\begin{definition}\label{complete-crv}
For every map $P \in \mathcal{P}$ we denote by $F_{P,m}$ the connected component indexed by $m$. 
\end{definition}

For the points with large starting energy $v_0>V_*$ we denote by 
$$
S_{-1, \V}=\bigcup_{m:\text{ }\mathcal{U}_{en,m}^{s} \subset [0,2]\times (\V, \infty) }\mathcal{U}_{en,m}
$$
and $S_{+1}$ the compliment of $S_{-1, \V}$
$$
S_{+1, \V}=[0,2]\times (\V, \infty)\setminus S_{+1, \V}.
$$
The choice of $+1$ and $-1$ will be clear later. For ease of presentation, we may drop the term $\V$ in the notations $S_{+1, \V}$ and $S_{+1, \V}$.
Next we suppose the ball is in the upper left chamber $R_2^+$ in $(\tau,\mathcal{I})$-cylinder, then it has no choice but to return to the floor on the right (c.f. Figure \ref{fig:dividedphase} on the left). However, if we suppose the ball is in the lower left chamber $R_2^-$ in the $(\rho,\mathcal{J})$-cylinder, then it may return to the right floor following either the Long Exit for points in 
$$
\mathcal{L}_{ex}^l:=\{ \mathcal{M}_*\mathcal{J}(2+\zeta_1^* - \zeta_2^*) -\rho \}_2 \gtrsim 1 \}$$ or the Short Exit for points in 
$$\mathcal{L}_{ex}^s:= \{ \mathcal{M}_*\mathcal{J}(2+\zeta_1^* - \zeta_2^*) -\rho \}_2 \lesssim 1 \}$$ (c.f. Figure \ref{fig:dividedphase} on the right).\\

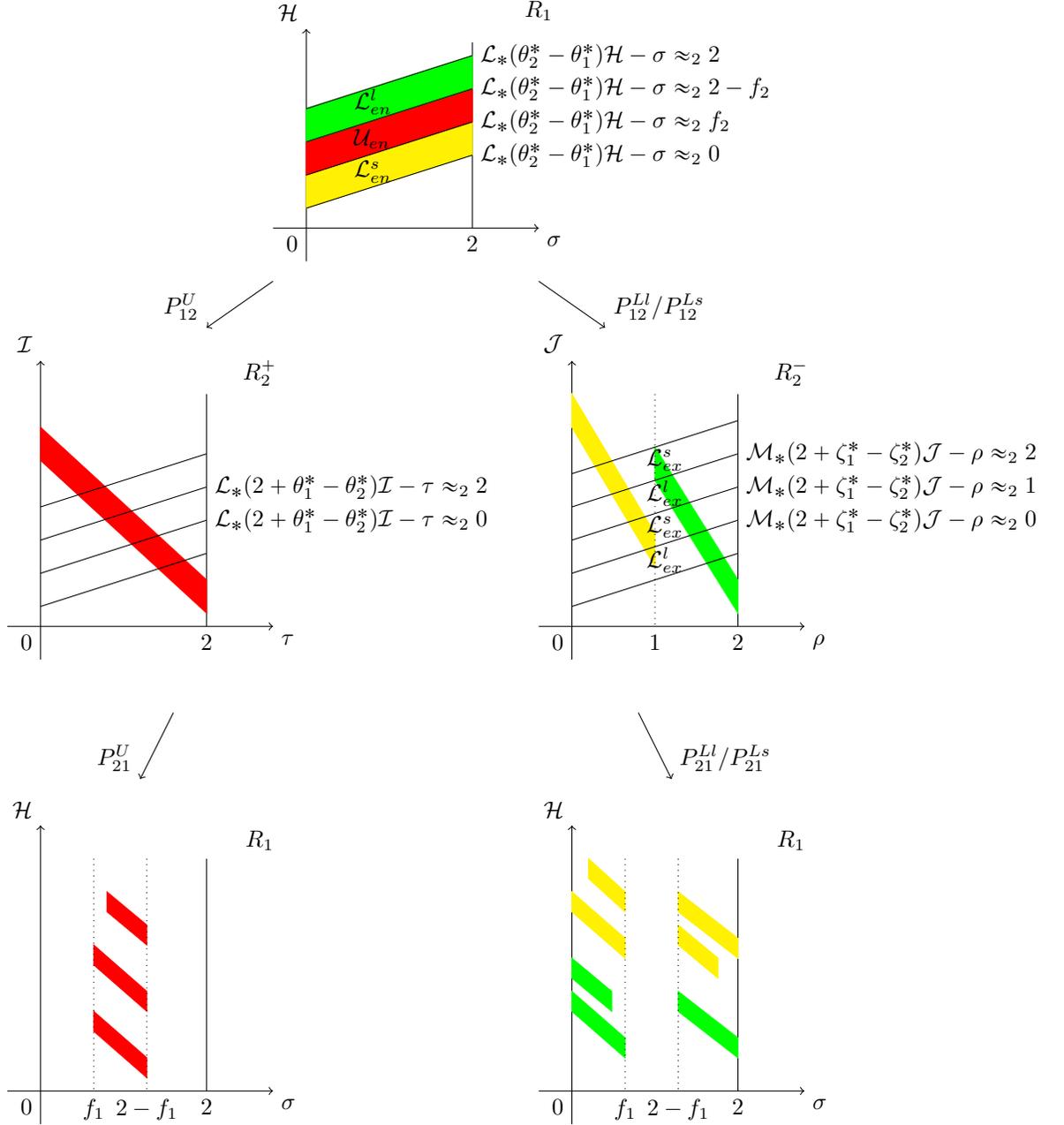
\begin{figure}[!ht]
    \centering
    \begin{tikzpicture}
       \draw[->] (-2,0)--(2,0) node[anchor=north west]{$\sigma$};
       \draw[->] (-1.5,-0.5)--(-1.5,3)node[anchor=south east]{$\mathcal{H}$};
       \draw (1,2.8)--(1,0) node[anchor=north]{2};
       \draw (-1.5,0) node[anchor=north east]{0}
             (2,3) node[anchor=south]{$R_1$};
             
       \filldraw[yellow] (-1.5,0.3) -- (1,1.1) -- (1,1.6) -- (-1.5,0.8) -- cycle;
       \filldraw[red] (-1.5,0.8) -- (1,1.6) -- (1,2.1) -- (-1.5,1.3) -- cycle;
       \filldraw[green] (-1.5,1.3) -- (1,2.1) -- (1,2.6) -- (-1.5,1.8) -- cycle;
       \draw (-0.9,0.85) node[anchor=west]{$\mathcal{L}_{en}^s$}
             (-0.9,1.35) node[anchor=west]{$\mathcal{U}_{en}$}
             (-0.9,1.9) node[anchor=west]{$\mathcal{L}_{en}^l$};
       
       \draw (-1.5,0.3) -- (1,1.1) node[anchor=west]{$\mathcal{L}_*(\theta_2^*-\theta_1^*)\mathcal{H}-\sigma\approx_2 0$}
             (-1.5,0.8) -- (1,1.6) node[anchor=west]{$\mathcal{L}_*(\theta_2^*-\theta_1^*)\mathcal{H}-\sigma\approx_2 f_2$}
             (-1.5,1.3) -- (1,2.1) node[anchor=west]{$\mathcal{L}_*(\theta_2^*-\theta_1^*)\mathcal{H}-\sigma\approx_2 2-f_2$}
             (-1.5,1.8) -- (1,2.6) node[anchor=west]{$\mathcal{L}_*(\theta_2^*-\theta_1^*)\mathcal{H}-\sigma\approx_2 2$};
             
       \draw[->] (-6,-6)--(-2,-6) node[anchor=north west]{$\tau$};
       \draw[->] (-5.5,-6.5)--(-5.5,-2)node[anchor=south east]{$\mathcal{I}$};
       \draw (-3,-2.5)--(-3,-6) node[anchor=north]{2}
             (-5.5,-6) node[anchor=north east]{0};
       \draw (-2.2,-2.5) node[anchor=south]{$R_2^+$};
       
       \filldraw[red] (-5.5,-3) -- (-5.5,-3.5) -- (-3,-5.8) -- (-3,-5.3) -- cycle;
       
       \draw (-5.5,-5.7) -- (-3,-4.9) 
             (-5.5,-5.2) -- (-3,-4.4)  node[anchor=west]{$\mathcal{L}_*(2+\theta_1^*-\theta_2^*)\mathcal{I}-\tau\approx_2 0$}
             (-5.5,-4.7) -- (-3,-3.9)  node[anchor=west]{$\mathcal{L}_*(2+\theta_1^*-\theta_2^*)\mathcal{I}-\tau\approx_2 2$}
             (-5.5,-4.2) -- (-3,-3.4);
             
       \draw[->] (2,-6)--(6,-6) node[anchor=north west]{$\rho$};
       \draw[->] (2.5,-6.5)--(2.5,-2)node[anchor=south east]{$\mathcal{J}$};
       \draw (5,-2.5) -- (5,-6) node[anchor=north]{2}
             (2.5,-6) node[anchor=north east]{0};
       \draw (5.8,-2.5) node[anchor=south]{$R_2^-$};
       
       \filldraw[yellow] (2.5,-2.5) -- (2.5,-3) -- (3.75,-5.1) -- (3.75,-4.6) -- cycle;
       \filldraw[green] (3.75,-3.3) -- (3.75,-3.8) -- (5,-5.8) -- (5,-5.3) -- cycle;
       
       \draw[dotted] (3.75,-2.5) -- (3.75,-6) node[anchor=north]{1};
       \draw (3.5,-4.5) node[anchor=west]{$\mathcal{L}_{ex}^s$}
             (3.5,-4) node[anchor=west]{$\mathcal{L}_{ex}^l$}
             (3.5,-3.5) node[anchor=west]{$\mathcal{L}_{ex}^s$}
             (3.5,-5) node[anchor=west]{$\mathcal{L}_{ex}^l$};
       
       \draw (2.5,-5.7) -- (5,-4.9) 
             (2.5,-5.2) -- (5,-4.4) node[anchor=west]{$\mathcal{M}_*(2+\zeta_1^*-\zeta_2^*)\mathcal{J}-\rho\approx_2 0$}
             (2.5,-4.7) -- (5,-3.9) node[anchor=west]{$\mathcal{M}_*(2+\zeta_1^*-\zeta_2^*)\mathcal{J}-\rho\approx_2 1$}
             (2.5,-4.2) -- (5,-3.4) node[anchor=west]{$\mathcal{M}_*(2+\zeta_1^*-\zeta_2^*)\mathcal{J}-\rho\approx_2 2$}
             (2.5,-3.7) -- (5,-2.9);
             
       \draw[->] (-6,-13)--(-2,-13) node[anchor=north west]{$\sigma$};
       \draw[->] (-5.5,-13.5)--(-5.5,-9)node[anchor=south east]{$\mathcal{H}$};
       \draw (-3,-9.5)--(-3,-13) node[anchor=north]{2};
       \draw (-2.2,-9.5) node[anchor=south]{$R_1$}
             (-5.5,-13) node[anchor=north east]{0};
       
       \filldraw[red] (-4.5,-10) -- (-4.5,-10.3) -- (-3.9,-10.8) -- (-3.9,-10.5) -- cycle
                      (-4.7,-10.8) -- (-4.7,-11.1) -- (-3.9,-11.8) -- (-3.9,-11.5) --  cycle
                      (-4.7,-11.8) -- (-4.7,-12.1) -- (-3.9,-12.8) -- (-3.9,-12.5) --  cycle;
       
       \draw[dotted] (-4.7,-9.5) -- (-4.7,-13) node[anchor=north]{$f_1$}
                     (-3.9,-9.5) -- (-3.9,-13) node[anchor=north]{$2-f_1$};
             
       \draw[->] (2,-13)--(6,-13) node[anchor=north west]{$\sigma$};
       \draw[->] (2.5,-13.5)--(2.5,-9)node[anchor=south east]{$\mathcal{H}$};
       \draw (5,-9.5) -- (5,-13) node[anchor=north]{2};
       \draw (5.8,-9.5) node[anchor=south]{$R_1$}
             (2.5,-13) node[anchor=north east]{0};
       
       \filldraw[yellow] (2.75,-9.5) -- (2.75,-9.8) -- (3.3,-10.3) -- (3.3,-10) -- cycle
                         (2.5,-10) -- (2.5,-10.3) -- (3.3,-11) -- (3.3,-10.7) -- cycle
                         (4.1,-10) -- (4.1,-10.3) -- (5,-11) -- (5,-10.7) -- cycle
                         (4.1,-10.5) -- (4.1,-10.8) -- (4.7,-11.3) -- (4.7,-11) -- cycle;
        \filldraw[green] (2.5,-11) -- (2.5,-11.3) -- (3.1,-11.8) -- (3.1,-11.5) -- cycle
                         (2.5,-11.5) -- (2.5,-11.8) -- (3.3,-12.5) -- (3.3,-12.2) -- cycle
                         (4.1,-11.5) -- (4.1,-11.8) -- (5,-12.5) -- (5,-12.2) -- cycle;
       
       \draw[dotted] (3.3,-9.5) -- (3.3,-13) node[anchor=north]{$f_1$}
                     (4.1,-9.5) -- (4.1,-13) node[anchor=north]{$2-f_1$};
                     
        \draw[->] (-2,-0.8) -- (-3,-1.5) node[anchor=south east]{$P_{12}^U$};
        \draw[->] (2,-0.8) -- (3,-1.5) node[anchor=south west]{$P_{12}^{Ll}$/$P_{12}^{Ls}$};
        \draw[->] (-3.5,-7.3) -- (-4,-8.3) node[anchor=south east]{$P_{21}^U$};
        \draw[->] (3.5,-7.3) -- (4,-8.3) node[anchor=south west]{$P_{21}^{Ll}$/$P_{21}^{Ls}$};
    \end{tikzpicture}
    \caption{The Divided Phase Cylinder}
    \label{fig:dividedphase}
\end{figure}

When jumping from one strip in the phase cylinder to another, an ascending (viewed from left to right) strip is mapped into a narrow but long descending (viewed from left to right) strip. More precisely, the top boundary
$\{ \mathcal{L}_*\mathcal{H}(\theta_2^* - \theta_1^*) -\sigma \}_2 \approx 2-f_2$ and the bottom boundary $\{ \mathcal{L}_*\mathcal{H}(\theta_2^* - \theta_1^*) -\sigma \}_2 \approx f_2$ of $\mathcal{U}_{en}$ are mapped by $P_{12}^U$ into the left boundary $\{\tau=0\}$ and the right boundary $\{\tau=2\}$ respectively, etc. (c.f. Figure \ref{fig:dividedphase}).

From the discussion in Section \ref{normal} half-revolution $P\in\cP$ from one strip to another is decomposed into a linear part $G$, a first-order correction $H$ and a second-order error term $\mathcal{O}_3$:
\[\label{cnt}
   P=G+H+\mathcal{O}_3.
\]

It follows from an easy check of calculus that the maps $P$ are $C^2$ perturbations of their linear parts $G$ for sufficiently large initial energies $\mathcal{H}>V_*$; this observation plays a vital role in the analysis in the next Section.

\medskip

Now we describe how the energy of the ball changes in a period when it follows two different routes.

For a point $\varrho=(x_0,y_0) \in R_1$ we denote (with an abuse of notation) by $y_n=P^n y_0$ the energy of the point $\varrho$ at the $n$-th half-revolution and $z_n=\cf^n y_0$ the energy at the $n$-th complete revolution, i.e. $z_n=y_{2n}$.

\begin{prop}\label{M00}
There exists $\ell<1$ so that for sufficiently large $\V=\V(f_1, f_2, \dot f_1,\dot f_2)$ we have for every $(x_0, z_0)\in S_{+1, \V}$ that
\begin{equation}
\ell\frac{f_2}{f_1} z_0\leq z_1
\end{equation}
and respectively for $(x_0, z_0)\in S_{-1, \V}$
\begin{equation}
\ell\left(\frac{1-f_2}{1-f_1}\right) z_0 \leq z_1. 
\end{equation}
Under the assumption $f_2>f_1$ and for an appropriate choice of $\ell$ and $\V$ we will have that for $(x_0,z_0)\in S_{+1}$, $z_0< z_1$ and for $(x_0,z_0)\in S_{-1}$ $z_0> z_1$.
Then, for every $(x_0,y_0) \in S_{+1, \V}\cup S_{-1, \V}$
\begin{equation}\label{intrmd}
\min\left\{ \frac{f_2}{2} y_0,\frac{\left(1-f_2\right)}{2} y_0 \right\}
\leq y_1.
\end{equation}
and for some $c<\infty$ we have for all $(x_0, z_0) \in S_{+1, \V} \cup S_{-1, \V}$
\begin{equation}
\left|\frac{\ln z_n - \ln z_0}{n}\right|\leq c.
\end{equation}
\end{prop}
\begin{proof}

We observe that in the definition of the Poincare maps $P$ the parameters $\Delta_1 , \Delta_2, \kappa_l,$ etc. do not depend on the energy of the ball, and hence it follows that for some number $D=D(\dot f_1, \dot f_2)>0$ for all $\varrho=(x_0, y_0) \in S_{+1,\V}$, we have for $P_{12}^{Ll/Ls}$ (lower route) that
\begin{equation}\label{F1}
f_2 y_0 -D \leq y_1 \leq f_2 y_0 + D,
\end{equation}
and for $P_{21}^{Ll/Ls}$
\begin{equation}\label{F2}
\frac{1}{f_1} y_0 -D \leq y_1 \leq \frac{1}{f_1} y_0 + D.
\end{equation}
And respectively for $P_{12}^{U}$ (upper route) and every $(x_0, y_0) \in S_{-1}$
\begin{equation}\label{F3}
\left(1-f_2\right) y_0-D \leq y_1\leq \left(1-f_2\right) y_0+D,
\end{equation}
and for $P_{21}^{U}$
\begin{equation}\label{F4}
\left(\frac{1}{1-f_1}\right) y_0-D \leq y_1\leq \left(\frac{1}{1-f_1}\right) y_0+D.
\end{equation}
After a complete revolution, when the particle returns to the $R_1$ cylinder, we will have that for any $(x_0, y_0) \in S_{+1}$
$$
\left(\frac{f_2}{f_1}\right) y_0-D_0 \leq y_2\leq \left(\frac{f_2}{f_1}\right) y_0+D_0,
$$
and respectively for the upper route $(x_0, y_0) \in S_{-1}$
$$
\left(\frac{1-f_2}{1-f_1}\right) y_0-D_0 \leq y_2\leq \left(\frac{1-f_2}{1-f_1}\right) y_0+D_0,
$$
where $D_0=D_0(f_1, f_2, \dot f_1,\dot f_2)>0$. Now let $\ell<1$ be very close to $1$. Note that if $\V$ is large enough then for all $(x_0, y_0)\in S_{+1}$, we will have that
\begin{equation}\label{R1}
\ell\frac{f_2}{f_1} y_0\leq y_2 \leq \frac{1}{\ell}\frac{f_2}{f_1} y_0
\end{equation}
and respectively for the lower route
\begin{equation}\label{R2}
\ell\left(\frac{1-f_2}{1-f_1}\right) y_0 \leq y_2\leq \frac{1}{\ell}\left(\frac{1-f_2}{1-f_1}\right) y_0.
\end{equation}

This proves the first two statements of the Proposition. 
We have already assumed that
\begin{equation}\label{asmpt}
f_2 > f_1
\end{equation}
We take $\ell$ so close to $1$ that 
$$
\ell \frac{f_2}{f_1}>1
$$
and
$$
\ell\left(\frac{1-f_2}{1-f_1}\right)<1.
$$
Note that under these assumptions, for $(x_0, z_0) \in S_{+1}$ we will have for the upper route that
$$
z_0> z_1
$$
and for the lower root
$$
z_0< z_1.
$$
Thus, under the assumption $f_2 > f_1$ the lower route accelerates (i.e. $\varrho \in S_{+1}$) while the upper route deaccelerates (respectively $\varrho \in S_{-1}$).

Since
$$
\left(\frac{1-f_2}{1-f_1}\right)<\frac{f_2}{f_1},
$$
then by \eqref{R1} and \eqref{R2}, for every $(x_0, z_0)\in S_{+1} \cup S_{-1}$ we have the following universal upper bound
$$ 
z_1 \leq z_0 \frac{1}{\ell}\frac{f_2}{f_1}.
$$
Iterating this we obtain
$$
z_n\leq z_0\left(\frac{1}{\ell}\frac{f_2}{f_1}
\right)^n .
$$
Hence
\begin{equation}\label{I1}
\left|\frac{\ln z_n - \ln z_0}{n}\right|\leq \ln \frac{1}{\ell}\left(\frac{f_2}{f_1}\right)=c<\infty.
\end{equation}

To obtain a universal lower bound, note that by \eqref{F1}, \eqref{F3} for every point $(x_0, y_0)\in S_{+1} \cup S_{-1}$ we have that
\begin{equation}
\min\{f_2 y_0 -D,\left(1-f_2\right) y_0-D \} \leq y_1.
\end{equation}
If $\V$ is large enough, then
$$
\min\{f_2 y_0 -D,\left(1-f_2\right) y_0-D \}\geq \min\left\{ \frac{f_2}{2} y_0,\frac{\left(1-f_2\right)}{2} y_0 \right\}.
$$
\end{proof}

%-------------------------------------------------------------
\section{Hyperbolicity}
In this section we show that under the assumptions in Theorem \ref{expAcc} our system is uniformly hyperbolic in the sense that all the six half-revolution maps $P$ share a common unstable invariant cone, provided that $\dot{f}_1,\dot{f}_2$ are sufficiently large. 

We recall that $P$ is a $C^2$ perturbation of its linear part $G$. We begin with the discussion of uniform cone hyperbolicity for the linear maps $G$ and then transfer it to $P$ through the robustness of uniform hyperbolicity under $C^2$-perturbations. We recall that the Long Enter and the Short Enter have the same derivatives for the linear parts, and so do the Long Exit and the Short Exit, thus we use the same notations for their derivatives respectively: $DG_{12}^L := DG_{12}^{Ll}=DG_{12}^{Ls}$ and $DG_{21}^L := DG_{21}^{Ll}=DG_{21}^{Ls}$.

More precisely, these derivatives take the following forms: 
\[ DG_{12}^U = 
   \begin{pmatrix}
      \frac{1}{l_2} & -\frac{\mathcal{L}_*(\theta_2^*-\theta_1^*)}{l_2}\\
      \frac{\Delta_2}{l_2} & l_2-\frac{\Delta_2\mathcal{L}_*(\theta_2^*-\theta_1^*)}{l_2}
   \end{pmatrix},
\]

\[ DG_{21}^U = 
   \begin{pmatrix}
      l_1 & -\mathcal{L}_*(2+\theta_1^*-\theta_2^*)l_1\\
      \Delta_1 l_1 & \frac{1}{l_1}-\Delta_1\mathcal{L}_*(2+\theta_1^*-\theta_2^*)l_1
   \end{pmatrix},
\]

\[ DG_{12}^L = 
   \begin{pmatrix}
      \frac{1}{f_2} & -\frac{\mathcal{L}_*(\theta_2^*-\theta_1^*)}{f_2}\\
      \frac{\kappa_l}{f_2} & f_2-\frac{\kappa_l\mathcal{L}_*(\theta_2^*-\theta_1^*)}{f_2}
   \end{pmatrix},
\]
and 
\[ DG_{21}^L = 
   \begin{pmatrix}
      f_1 & -\mathcal{M}_*(2+\zeta_1^*-\zeta_2^*)f_1\\
      \chi_l f_1 & \frac{1}{f_1}-\chi_l\mathcal{M}_*(2+\zeta_1^*-\zeta_2^*)f_1
   \end{pmatrix}.
\]

First we observe that $\det DG_{12}^U=\det DG_{21}^U=\det DG_{12}^L=\det DG_{21}^L=1$. We obtain from Pesin theory that the uniform hyperbolicity can be coined in the sense of invariant cones as follows.

Consider a map $T:M \righttoleftarrow$. 
We say $\{\mathcal{C}_u (\varrho)\}_{\varrho\in M}$ and $\{\mathcal{C}_s (\varrho)\}_{\varrho\in M}$ are a family of \emph{unstable invariant cones} and \emph{stable invariant cones} respectively for the map $T$ if there exist $\lambda>1$ such that $\forall d\varrho\in\mathcal{C}_u(\varrho)$
\[
   D_{\varrho}T(\mathcal{C}_u (\varrho)) \subseteq \mathcal{C}_u(\mathcal{C}_u (T\varrho)),\quad 
   |D_{\varrho}T(d\varrho)| \ge \lambda |d\varrho| 
\]
and $\forall d\varrho\in\mathcal{C}_s (\varrho)$
\[
   D_{\varrho} T^{-1}(\mathcal{C}_s (\varrho)) \subseteq \mathcal{C}_s (T^{-1}\varrho),\quad 
   |D_{\varrho} T^{-1}(d\varrho)| \ge \lambda |d\varrho|.
\]
 
For hyperbolic linear maps, they naturally possess invariant un/stable cones: suppose that $T$ is a hyperbolic linear map and that $\mathbf{e}_u, \mathbf{e}_s$ are its unstable and stable unit eigenvectors respectively, then for any $\varrho\in M$ and any $d\varrho\in T_{\varrho} \varrho$, $d\varrho=c_u \mathbf{e}_u + c_s \mathbf{e}_s$ for some $c_u,c_s$. It is easy to check that 
\[
   \mathcal{C}_{u,k}(\varrho) := \{d\varrho:|c_u|>k|c_s| \}
\]
defines a family indexed by $k>1$ of unstable cones. A family of stable cones can be defined similarly with the inequality reversed 
\[
   \mathcal{C}_{s,k}(\varrho) := \{d\varrho:|c_s|>k|c_u| \}.
\]
Suppose the unstable eigenvalue of $T$ is $\lambda$, then the expansion rate $\lambda_k$ of an unstable cone $\mathcal{C}_{u,k}$ with opening gauge $k$ is 
\begin{align*}
    \lambda_k^2
    &= \frac{|D_{\varrho}T(d\varrho)|^2}{|d\varrho|^2}\\
    &= \frac{c_u^2\lambda^2 + c_s^2/\lambda^2 + 2c_u c_s \cos<\mathbf{e}_u,\mathbf{e}_u>}{c_u^2 + c_s^2 + 2c_u c_s \cos<\mathbf{e}_u,\mathbf{e}_u>}\\
    &\ge \frac{c_u^2\lambda^2 + c_s^2/\lambda^2 - 2|c_u c_s|}{c_u^2 + c_s^2 + 2|c_u c_s|}\\
    &> \frac{(k\lambda-1/\lambda)^2}{(k+1)^2}
\end{align*}
where we have used the fact that $|c_u|> k|c_s|$ for $dx\in\mathcal{C}_{u,k}$. The number $\displaystyle \frac{k\lambda-1/\lambda}{k+1}$ can be easily made larger than $1$ by taking $\lambda$ sufficiently large (in fact, $\lambda>2$ would suffice).\\

Now we claim that our linear maps $G_{12}^U$, $G_{21}^U$, $G_{12}^L$, $G_{21}^L$ share a common unstable invariant cone provided that $\dot{f}_1,\dot{f}_2$ are sufficiently large. We prove the claim in two steps: first we show in Lemma \ref{hyperbolicmatrix} that the unstable eigenvector of a hyperbolic map are almost vertical and the stable eigenvector remains a positive angular distance from the unstable cones provided that the bottom entries are significantly larger than the top entries; then we show in Proposition \ref{gcone} that the almost vertical unstable cones of our maps $G_{12}^U$, $G_{21}^U$, $G_{12}^L$, $G_{21}^L$ have a nontrivial intersection, which is the common invariant unstable invariant cone we aim for.

\begin{lemma}\label{hyperbolicmatrix}
Let $A_n \in SL(2, \mathbb{R})$
\begin{equation*}
A_n= 
\begin{pmatrix}
a_n & b_n \\
c_n & d_n
\end{pmatrix},
\end{equation*}
with nonzero entries. Assume that $a_n,b_n$ are uniformly bounded, $b_n/a_n\to M$ for some constant $M\neq 0$ and $c_n,d_n\to\infty$ as $n\to\infty$. Then the unstable eigenvector $\mathbf{e}_u\to (0,1)$ and the stable eigenvector $\mathbf{e}_s\to (-M,1)/\sqrt{M^2 +1}$ as $n\to\infty$.
\end{lemma}

\begin{proof}
The unstable eigenvalue of the matrix $A_n$ is 
\[
   \lambda_n = \frac{(a_n+d_n)+\sqrt{(a_n+d_n)^2-4}}{2},
\]
and the eigenvectors (not necessarily unit vectors) are 
\[
   \mathbf{v}_u^n= (b_n, \lambda_n -a_n), \ 
   \mathbf{v}_s^n= (b_n, 1/\lambda_n -a_n).
\]
We note, as $n\to\infty$ that 
\[
   \frac{b_n}{\lambda_n -a_n}=\frac{2b_n}{d_n-a_n+\sqrt{(a_n+d_n)^2-4}}\to 0
\]
and that 
\[
   \frac{b_n}{1/\lambda_n -a_n} \to -M,
\]
which concludes our proof.
\end{proof}

\begin{prop}\label{gcone}
There exists $\dot{f}_*\gg 1$ such that if $|\dot{f}_1|,|\dot{f}_2|>\dot{f}_*$, then $G_{12}^U$, $G_{21}^U$, $G_{12}^L$, $G_{21}^L$ share a common invariant unstable cone $\mathcal{C}_u^*=\mathcal{C}_{u,\dot{f}_*}$ with the minimal expansion rate $\lambda_*=\lambda_{\dot{f}_*}$.
\end{prop}

\begin{proof}
First we note that the top entries of the matrices $DG_{12}^U$, $DG_{21}^U$, $DG_{12}^L$, $DG_{21}^L$ are of order $1$ and that the bottom entries contain $\Delta_2$, $\Delta_1$, $\kappa_l$ and $\chi_l$ respectively. We recall that 
\[
\Delta_1=-\frac{1}{2}\frac{\dot{l}_1}{l_1}, \ \Delta_2=\frac{1}{2}l_2\dot{l}_2, \ \kappa_l=\frac{1}{2}f_2\dot{f}_2, \ \chi_l=-\frac{1}{2}\frac{\dot{f}_1}{f_1},
\]
which can be made arbitrarily large by choosing $|\dot{f}_1|,|\dot{f}_2|\gg1$. Thus by Lemma \ref{hyperbolicmatrix} we know that there exists $\dot{f}_*\gg1$ such that if $|\dot{f}_1|,|\dot{f}_2|>\dot{f}_*$ the unstable eigenvectors of the four matrices $DG_{12}^U$, $DG_{21}^U$, $DG_{12}^L$, $DG_{21}^L$ are almost vertical while the stable eigenvectors tend to $(\mathcal{L}_*(\theta_2^*-\theta_1^*),1)$, $(\mathcal{L}_*(2+\theta_1^*-\theta_2^*),1)$, $(\mathcal{L}_*(\theta_2^*-\theta_1^*),1)$ and $(\mathcal{M}_*(2+\zeta_1^*-\zeta_2^*),1)$ respectively, all of which remain positive angular distance to the vertical direction. Therefore by choosing the opening gauge $k_i$ $(i=1,2,3,4)$ of cones carefully, the four unstable cones $\mathcal{C}_{u,k_i}$ of the four maps $G_{12}^U$, $G_{21}^U$, $G_{12}^L$, $G_{21}^L$ can have a nontrivial intersection $\mathcal{C}_u^*=\cap_i \mathcal{C}_{u,k_i}$ containing the four unstable eigenvectors of the four matrices as well as the vertical direction, and $\mathcal{C}_u$ remain a positive angular distance to the four stable directions of the four maps.

We claim that this nontrivial intersection $\mathcal{C}_u^*$ is invariant under the four maps $G_{12}^U$, $G_{21}^U$, $G_{12}^L$, $G_{21}^L$.
Indeed, for any $\varrho$ in the phase cylinder and any vector $d\varrho\in\mathcal{C}_u^*(\varrho)$, any map $G$ among the four maps $G_{12}^U$, $G_{21}^U$, $G_{12}^L$, $G_{21}^L$ will map $d\varrho$ closer to the corresponding unstable eigenvector $\mathbf{e}_u$, thus still remains in $\mathcal{C}_u$ (c.f. Figure \ref{fig:cone}).

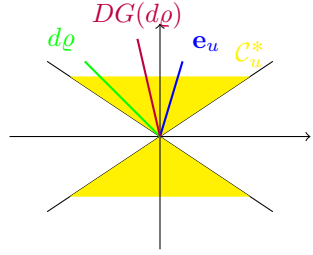
\begin{figure}[!ht]
    \centering
       \begin{tikzpicture}
          \draw[->] (-2,0)--(2,0);
          \draw (-1.5,1)--(1.5,-1)  (-1.5,-1)--(1.5,1);
          \fill[yellow] (-1.2,0.8)--(0,0)--(1.2,0.8)--cycle
                        (-1.2,-0.8)--(0,0)--(1.2,-0.8)--cycle;
          \draw[->] (0,-1.5)--(0,1.5);
          \draw[yellow] (1.2,0.8) node[anchor=south]{$\mathcal{C}_u^*$};
          \draw[blue,thick] (0,0)--(0.3,1) node[anchor=south west]{$\mathbf{e}_u$};
          \draw[thick,green] (0,0)--(-1,1) node[anchor=south east]{$d\varrho$};
          \draw[thick,purple] (0,0)--(-0.3,1.3) node[anchor=south]{$DG(d\varrho)$};
       \end{tikzpicture}
    \caption{The common unstable cone $\mathcal{C}_u^*$}
    \label{fig:cone}
\end{figure}

$\mathcal{C}_u^*$ guarantees an expansion rate at least the minimum of the four expansion rates of the four unstable cones: $\lambda_*=\min_{i}\{\lambda_{k_1},\lambda_{k_2},\lambda_{k_3},\lambda_{k_4}\}$.
\end{proof}

Finally, the Poincar\'e maps $P$ are $C^1$ perturbations of their linear parts $G$  if the initial energy are sufficiently large. The classical literature on the robustness of uniform hyperbolicity under $C^1$ perturbations (for example c.f. Appendix A in \cite{Viana97}) guarantees that the map $P$ inherits the invariant unstable cone with the desired properties in Proposition \ref{gcone}. More precisely, with an abuse of notation on the unstable cones, we have the following
\begin{prop}[Invariant Cone]\label{pcone}
Assume that $|\dot{f}_1|,|\dot{f}_2|>\f$. There exists $V_*\gg 1$ such that if the initial energy $v_0>V_*$, then $P_{12}^U$, $P_{21}^U$, $P_{12}^L$, $P_{21}^L$ share a common invariant unstable cone $\mathcal{C}_u^*$ and the minimal expansion rates for $P_{12}^U$, $P_{21}^U$, $P_{12}^L$ and $P_{21}^L$ in the cone $\mathcal{C}_u^*$ are $\lambda_{12}^U$, $\lambda_{21}^U$, $\lambda_{12}^L$ and $\lambda_{21}^L$ respectively.
\end{prop}

\begin{rmk}\label{lbd-f}
We note that there are constants $c_1,c_2>0$, independent of $\dot f_1, \dot f_2$, so that for each map $P \in \{P_{12}^U,P_{12}^{Ll},P_{12}^{Ls}\}$ we have that
$$
c_1 |\dot f_2|\leq \lambda_{P}\le\Lambda_P\leq c_2 |\dot f_2|,
$$
and respectively for every $P\in \{P_{21}^U,P_{21}^{Ll},P_{21}^{Ls}\}$
$$
c_1 |\dot f_1|\leq \lambda_{P}\le\Lambda_P\leq c_2 |\dot f_1|,
$$
where $\lambda_{P}$ and $\Lambda_P$ are the minimal and maximal expansion rates of the map $P$ in the unstable cone $\mathcal{C}_u^*$ respectively. Also we denote the maximal and minimal expansion of $\cf$ by 
\[
   \lambda_{\cf}=\min\{\lambda_{12}^U\lambda_{21}^U,\lambda_{12}^L\lambda_{21}^L\},\quad \Lambda_{\cf}=\max\{\Lambda_{12}^U\Lambda_{21}^U,\Lambda_{12}^L\Lambda_{21}^L\}.
\]
\end{rmk}

\begin{definition}[Unstable curves]\label{unstablecurve}
We say $\gamma$ is an \emph{unstable curve} if it is $C^2$, it lies above the threshold $\V$ and the slope of any point in $\gamma$ lies in the unstable cone, i.e. $\forall \varrho=(x,y) \in \gamma$, $\mathcal{J}_\gamma \varrho  \in  \mathcal{C}_{u}^*$ where $\mathcal{J}_\gamma$ is the directional derivative.
\end{definition}

\begin{comment}
We denote by $\Gamma_{\V}$ the collection of all the unstable curves with starting energy $v_0>V_*$.
\end{comment}

%-------------------------------------------------------------

\section{Auxiliary construction: The Modified System}
 
We remind the reader that the normal forms and all the estimates obtained in the previous sections are valid only for large values of $\V$. Therefore, when the energy drops below $\V$ our analysis may collapse. In this section, we modify our dynamics near the threshold $V_*$ in a way that whenever the ball's energy falls near $V_*$, we push the energy of ball up in the next period.

%Later on we will show that if the particle starts its motion with a sufficiently large initial energy $z_0$ then its energy will remain above the level $\V$ with a large probability. Hence, the definition of the maps below $\V$ will not be important since the trajectories that will go below it will be discarded. However, to avoid technical difficulties connected with the energy of the particle going below $\V$, 

Now we introduce the modified system: we assume that when the energy of the particle approaches $\V$, then the particle will be forced into the lower route which will increase the energy. More precisely, we choose a number $V_0 \gg \V$ and for all $m<m_*$ (those for which $\mathcal{U}_{en,m}$ entirely lies below $V_0$) we replace the definition of $P$ by the linear parts $G$ of accelerating lower route and hence we obtain a new map $P_0$ as follows  
\begin{equation}\label{Mdf}
P_0(x,y)=\begin{cases} G_{12}^{Ls/Ll}(x,y), & (x,y)\in \mathcal{U}_{en,m}, \text{ for }m \text{,  with }\\
&\mathcal{U}_{en,m}\subset \Big(R_1\cap [0, V_0)\Big), \\ P(x,y) ,& \text{otherwise}.\end{cases}
\end{equation}

Note that $P_0$ will have the same singularity lines in $[0,V_0]$ as $P$. The corresponding connected components will still be considered parts of sets $S_{+1}$ and $S_{-1}$. Thus, although some of the connected components in $S_{-1}$ the once below $V_0$, will actually increase the energy of the ball, but for convenience we will continue to consider them as part of $S_{-1}$.

Clearly, if $V_0$ is large enough compared with $\V$ then due to the lower bound \eqref{intrmd}, the energy can not go too low from $V_0$ as it will start to increase once its below $V_0$ and it will eventually be pushed back above $V_0$ by the dynamics.

The modified map $P_0$ inherits nice properties from the original map $P$ and we summarize them in the following Proposition.

\begin{prop}\label{M0}
Assume that $\f$ is sufficiently large and we have $f_2>f_1$. Then there exist constants $\V, V_0, \ell, c$, depending on the parameters $f_1,f_2,\dot f_1,\dot f_2$, so that the modified system $P_0$, defined in \eqref{Mdf}, satisfies the properties listed below: 

1) $P_0(x,y)=P(x,y)$, for all $(x, y)\notin \Big([0,2]\times [0, V_0)\Big)$. $P_0$ is $C^2$ on $\Big([0,2]\times [\V, V_0)\Big)$, it preserves the unstable cones $\mathcal{C}_u^*$ and it has the same singularity lines and expansion rates (in $\mathcal{C}_u^*$) as $P$. In particular, Remark \ref{lbd-f} is full-filled for $P_0$.

2) $\cf\Big([0,2]\times [\V, \infty)\Big)\subset \Big([0,2]\times [\V, \infty)\Big)$

3) For every $(x_0, z_0)\in S_{+1}$ we have that
\begin{equation}\label{R10}
\ell\frac{f_2}{f_1} z_0\leq z_1
\end{equation}
and respectively for $(x_0, z_0)\in S_{-1}$
\begin{equation}\label{R20}
\ell\left(\frac{1-f_2}{1-f_1}\right) z_0 \leq z_1. 
\end{equation}
Since $f_2>f_1$, then $\ell\left(\frac{1-f_2}{1-f_1}\right) <1<\ell\frac{f_2}{f_1}$ for $\ell$ close to $1$.

4) There exists $c<\infty$ so that for all $(x_0, z_0) \in S_{+1} \cup S_{-1}$
\begin{equation}\label{bound}
\left|\frac{\ln z_n - \ln z_0}{n}\right|\leq c.
\end{equation}
\end{prop}
\begin{proof}
The first statement in 1) follows from the definition of $P_0$. The second statement in 1) follows from the fact that we have changed the definition of $P$ on entire connected component. Hence $P_0$ will still be smooth on $S_{+1}$ or $S_{-1}$ and will have the same singularity lines as $P$. Cone invariance will still hold since we have it for $P$ above $\V$ and it is immediate for the linear maps $G_{12}^{Ls/Ll}$. $P_0$ coincided with the linear parts which obviously inherit the expansion rates from the original system, i.e. we have the bounds in Remark \ref{lbd-f}.

For $S_{+1}$ and $S_{-1}$ above $V_0$, 3) and 4) follow from Proposition \ref{M00}.
Since $\ell\left(\frac{1-f_2}{1-f_1}\right)<1$ and in the fundamental components in $[0,2]\times [\V, V_0]$ we either have the lower bound $\ell\left(\frac{1-f_2}{1-f_1}\right)z_0$ or $\ell\frac{f_2}{f_1} z_0$. Hence, in both cases
 \eqref{R20}, clearly holds.

2) follows from the fact that near $\V$ the energy is only allowed to accelerate after a full revolution.
\end{proof}

\begin{definition}
We will denote by $\cf_0=P_0^{2}$ the map associated with the complete revolution.
\end{definition}

Our strategy to prove exponential acceleration for the modified system is as follows: we will increase $\dot f_*$ (recall that $|\dot f_1|, |\dot f_2|>\f$) while keeping $f_1, f_2$ fixed. As $\f$ increases we will also increase the threshold $\V$ so that the properties listed in Proposition \ref{M0} are full-filled.

Note that while  $f_1, f_2$ are fixed then the distance between singularity lines will remain the "same" up to a small distortion (see Figure \ref{fig:dividedphase}). The lines may change their angle with respect to the vertical axes, but they will always maintain a uniform angle from it due to the assumption in the main Theorem.

\section{Growth Lemmas and deviation estimates}

In this section we lay down some important outcomes of the uniform cone hyperbolicity of our system: the growth lemmas and a deviation estimate. Although our system enjoys strong uniform hyperbolicity under the assumptions in Theorem \ref{expAcc}, the singularity curves on the divided phase might cut short unstable curves as we iterate and unstable curves might become and even remain short for a long time. However, the growth lemmas, which precisely describe the size and growth of long/short unstable curves, guarantee that the situation is not so hopeless.

The Section 6 is organized as follows. In Section \ref{grwth} we prove the main growth lemmas. We introduce the notion of a long curve, which refers to curves that are of a sufficiently large size (to be defined in Def. \ref{longcurve}).
We provide explicit estimates on the size of long curves  and the constants involved in the estimates, depending on the expansion rate of the system.

In Section \ref{delay} we prove a growth lemma which has a certain delay: namely, we first iterate the unstable curve $N_0$ many times and then start measuring the size and the growth of unstable curves. The purpose of the parameter $N_0$ will be clear in Section \ref{fnt-time}.

In Section \ref{dv-est}, for each initial condition $\varrho$ in a long curve, we consider the times $\{n_k\}$ when $\cf_0^{n_k} \varrho$ again belongs to a long curve. Proposition \ref{dev-prop} provides large deviation bounds for the growth rate of these times.

We refer the reader to \cite{CM,CZ} for background and further results on hyperbolic systems with singularities, where the growth lemmas play a vital role.

\begin{prop}[Distortion Control]\label{dstr-bd}
%Let $P_1, P_2, \dots, P_n \in \mathcal{P}$ be a sequence of transformations and $P_{1,n}=P_n\circ\cdots\circ P_1$. Let $n\in\mathbb{N}$ and $\gamma=\{(\tau, I):\tau=\tau_0\}\in \mathcal{C}_u$. Suppose $\Delta, J\subset \gamma$ two connected components of $P_{1,n}$.
Suppose that $I, J\subset \gamma$ are two connected components of $P_0^n \gamma$. Then there exists $K>1$, so that
$$
\frac{1}{K} \frac{|I|}{|J|}\leq \frac{|P_0^n I|}{|P_0^n J|} \leq K \frac{|I|}{|J|},
$$
and $\lim_{V_* \rightarrow \infty}K=1$, where $V_*$ is the critical threshold of starting energy and $|.|$ is the curve length.
\end{prop}

\begin{proof}
It suffices to show that for any $\varrho_n,\varrho'_n$ in a connected component $\tilde{\gamma}_n$ of $\gamma_n=P_0^n \gamma$, we have the following distortion control 
\[
   \left|\log\mathcal{J}_{\gamma_n}P_0^{-n}(\varrho_n) -\log\mathcal{J}_{\gamma_n}P_0^{-n}(\varrho'_n) \right| \le C|\varrho_n-\varrho'_n|,
\]
where $\mathcal{J}_{\gamma_n}P_0^{-n}(\varrho)$ means the Jacobian of $P_0^{-n}$ restricted to $\gamma_n$.

We observe that 
\begin{align*}
    &\quad \left|\log\mathcal{J}_{\gamma_n}P_0^{-n}(\varrho_n) -\log\mathcal{J}_{\gamma_n}P_0^{-n}(\varrho'_n) \right|\\
    &\le \sum_{m=0}^{n-1} \max_{\varrho_m\in\tilde{\gamma}_n} \left| \frac{d}{d\varrho_m} \log\mathcal{J}_{\gamma_m}P_0^{-1}(z_m) \right| |\varrho_m-\varrho'_m|\\
    &\le C \sum_{m=0}^{n-1} \frac{|\varrho-\varrho'|}{\Lambda_{P_0}^m}\\
    &\le C' |\varrho-\varrho'|
\end{align*}
where we have used the fact that $P_0$ has bounded derivatives. 
We also note that the constant $C'$ in the last line can be made as close to 0 (henceforth $K=e^{C'}$ as close to 1) as possible by taking the initial energy large, as $P_0$ are $C^2$ perturbations of linear maps $G$. 
\end{proof}

\subsection{Growth Lemmas}\label{grwth}

We denote the \emph{complexity} by $\kappa_n(\delta)$, i.e. the maximal number of pieces an unstable curve $\gamma$ can be cut into under the map $P_0^n \gamma$. It is easy to see from Figure \ref{fig:dividedphase} that for any $\delta_0<\min\{1/\Lambda_{12}^U,1/\Lambda_{12}^L\}$ and $\min\{\lambda_{12}^U\lambda_{21}^U,\lambda_{12}^L\lambda_{21}^L\}>4$ we have  $$\kappa_{1}(\delta_0)<\min\{\lambda_{12}^U\lambda_{21}^U,\lambda_{12}^L\lambda_{21}^L\}.$$
Indeed, for an unstable curve of size $\delta_0<\min\{1/\Lambda_{12}^U,1/\Lambda_{12}^L\}$ in $R_1$, it meets at most one singularity curve hence it gets cut into at most two pieces when mapped to $R_2$, each of size at most $\Lambda_{12}^U\delta_0$ and $\Lambda_{12}^L\delta_0$ respectively. Each of the pieces in $R_2$ meets at most one singularity curve due to our choice of $\delta_0$ hence it gets further cut into at most two pieces when mapped to $R_1$. Therefore, for such choice of $\delta_0$, the unstable curve gets cut into at most 4 pieces when finishing a complete period, and consequently if we choose $\dot{f}_*$ so large that the expansions satisfy $\min\{\lambda_{12}^U\lambda_{21}^U,\lambda_{12}^L\lambda_{21}^L\}>4$, we obtain the desired complexity control.

\begin{figure}[!ht]
    \centering
    \begin{tikzpicture}
       \draw[->] (-0.3,0) -- (2.5,0);
       \draw[->] (0,0) -- (0,3);
       \draw (0,0) node[anchor=north]{0}
             (2,0) node[anchor=north]{2};
       \draw (2,0) -- (2,2.8) node[anchor=south west]{$R_1$};
       \draw (0,0.3) -- (2,1)
             (0,0.8) -- (2,1.5)
             (0,1.3) -- (2,2)
             (0,1.8) -- (2,2.5);
       \draw[<->] (2.15,1.05) -- (2.15,1.45) node[anchor=north west]{$f_2$};
       \draw[<->] (2.15,1.55) -- (2.15,1.95) node[anchor=north west]{$2-2f_2$};
       \draw[<->] (2.15,2.05) -- (2.15,2.45) node[anchor=north west]{$f_2$};
       \draw[thick,red] (0.8,1.3) -- (1,0.8);
       
       \draw[->] (4,1.5) -- (5,1.5) node[anchor=south east]{$P_{12}^{Ll/s}$};
       
       \draw[->] (5.7,0) -- (8.5,0);
       \draw[->] (6,0) -- (6,3);
       \draw (6,0) node[anchor=north]{0}
             (7,0) node[anchor=north]{1}
             (8,0) node[anchor=north]{2};
       \draw (8,0) -- (8,2.8) node[anchor=south west]{$R_2^-$};
       \draw[dotted] (7,0) -- (7,2.5);
       \draw (6,0.3) -- (8,1)
             (6,0.8) -- (8,1.5)
             (6,1.3) -- (8,2)
             (6,1.8) -- (8,2.5);
       \draw[<->] (8.15,1.05) -- (8.15,1.45) node[anchor=north west]{1};
       \draw[<->] (8.15,1.55) -- (8.15,1.95) node[anchor=north west]{1};
       \draw[<->] (8.15,2.05) -- (8.15,2.45) node[anchor=north west]{1};
       \draw[thick,red] (6.8,1.3) -- (7,0.8)
                        (6,1.7) -- (6.3,1.2);
       
       \draw[->] (9,1.5) -- (10,1.5) node[anchor=south east]{$P_{21}^{Ll/s}$};
       
       \draw[->] (10.7,0) -- (13.5,0);
       \draw[->] (11,0) -- (11,3);
       \draw (11,0) node[anchor=north]{0}
             (11.5,0) node[anchor=north]{$f_1$}
             (12.3,0) node[anchor=north]{$2-f_1$}
             (13,0) node[anchor=north]{2};
       \draw (13,0) -- (13,2.8) node[anchor=south west]{$R_1$};
       \draw[dotted] (11.5,0) -- (11.5,2.5)  (12.3,0) -- (12.3,2.5);
       \draw[thick,red] (11,2)--(11.2,1.7) (11.3,1.4)--(11.5,1.1) (12.3,0.7)--(12.5,0.5) (12.7,1.8)--(13,1.5);
    \end{tikzpicture}
    \caption{Complexity Control}
    \label{complexity}
\end{figure}
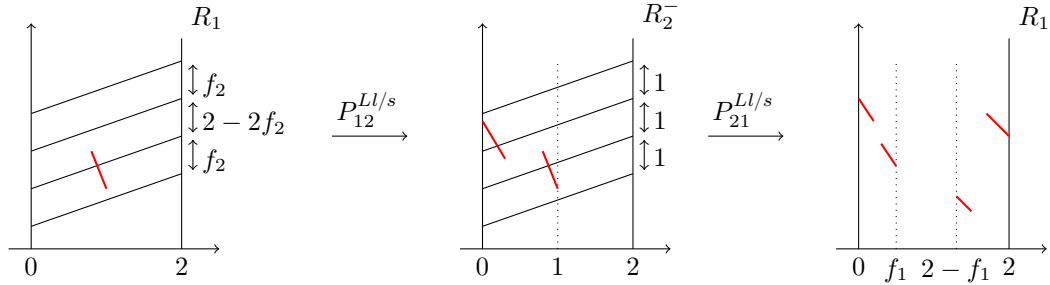

Let $\gamma$ be an unstable curve and $\varrho\in\gamma$. We denote by $r_n(\varrho)$ the distance from $\mathcal{F}_0^n(\varrho)$ to the boundary of the connected component $\gamma_n\subseteq \mathcal{F}_0^n(\gamma)$ containing $\varrho_n=\mathcal{F}_0^n(\varrho)$.

\begin{lemma}[First Growth Lemma]\label{1growthlemma}
There exist $\vartheta_1 :=\frac{\kappa_1 K^2}{\lambda_{\min}}<1$ and $C_2 := 2K^2/\delta_0(1-\vartheta_1)>0$ such that for any $\epsilon>0$
\[
  m_{\gamma}(r_n(\varrho)<\epsilon) \le \vartheta_1^n m_{\gamma}(r_0(\varrho)<\epsilon) + C_2 \epsilon |\gamma|
\]
where $m_{\gamma}$ is the Lebesgue measure restricted to $\gamma$.
\end{lemma}

\begin{proof}
We prove a slightly stronger result. We need to control the number of pieces while iterating an unstable curve, thus we cut a long curve into several pieces of length equal to or shorter than $\delta_0$, which comes from the complexity. We update to $r_n^*$ the distance from $\varrho_n$ to the real or artificial boundary (the latter introduced by the chopping procedure).

We claim that 
\[
   m_{\gamma}(r_{n+1}^*<\epsilon) \le \frac{\kappa_1 K^2}{\lambda_{\min}} m_{\gamma}(r_n^*<\epsilon) + \frac{2K^2\epsilon}{\delta_0} |\gamma|
\]
where $\kappa_1=\kappa_1(\delta_0)$ is the complexity and $K$ comes from the distortion in Proposition \ref{dstr-bd}.

Indeed, if $x_{n+1}$ is $\epsilon$-close to real or artificial boundary, the former contributes $\kappa_1 m_{\gamma}(r_n^*<\epsilon/\lambda_{\min})$ and the latter contributes at least $\displaystyle\left[\frac{|\gamma_{n+1}|}{\delta_0}\right]\frac{2K^2\epsilon|\gamma|}{|\gamma_{n+1}|}$. Therefore 
\begin{align*}
    m_{\gamma}(r_{n+1}^*<\epsilon) &\le \kappa_1 m_{\gamma}(r_n^*<\epsilon/\lambda_{\min}) + \left[\frac{|\gamma_{n+1}|}{\delta_0}\right]\frac{2K^2\epsilon|\gamma|}{|\gamma_{n+1}|}\\
    &\le \frac{\kappa_1 K^2}{\lambda_{\min}} m_{\gamma}(r_n^*<\epsilon) + \frac{2K^2\epsilon}{\delta_0}|\gamma|.
\end{align*}

Inductively, we have 
\[
   m_{\gamma}(r_n^*<\epsilon) \le \left(\frac{\kappa_1 K^2}{\lambda_{\min}}\right)^n (r_0^*<\epsilon) + \frac{2K^2\epsilon|\gamma|}{\delta_0} \left(1+\cdots+\left(\frac{\kappa_1 K^2}{\lambda_{\min}}\right)^{n-1}\right).
\]
We conclude the proof by taking $K$ sufficiently close to 1 such that $\vartheta_1 :=\frac{\kappa_1 K^2}{\lambda_{\min}}<1$ and $C_2 := 2K^2/\delta_0(1-\vartheta_1)$.
\end{proof}

As a consequence, after waiting long enough, we have the second growth lemma for short $\gamma$
\begin{lemma}[Second Growth Lemma]\label{2growthlemma}
For any unstable curve $\gamma$ and any $\epsilon>0$, there exists $C_3$ such that for $n>\log|\gamma|/\log\vartheta_1$  
\[
   m_{\gamma}(r_n<\epsilon)\le C_3 \epsilon |\gamma|.
\]
\end{lemma}
\begin{proof}
It is easy to see that $\vartheta_1^n<|\gamma|$ for $n>\log|\gamma|/\log\vartheta_1$.

Also, $m_{\gamma}(r_n<\epsilon)\le\min\{2\epsilon,|\gamma|\}$. 
Thus the second growth lemma follows from the first growth lemma with $C_3:=1+C_2$.
\end{proof}

Finally we prove that a short piece from an unstable curve cannot remain too short for a long time.
\begin{lemma}[Third Growth Lemma]\label{3growthlemma}
There exist $\epsilon_0,b_1,\vartheta_2<1$ such that for any $k\in\mathbb{N}$ $$\min\{\vartheta_1^{k -1},\epsilon_0\}\le b_1\vartheta_2^{k}$$ and that $\vartheta_3:=\vartheta_2(1+C_3 Kb_1)<1$.
Then for any $n_2>n_1>\log|\gamma|/\log\vartheta_1$ 
\[
   m_{\gamma}\left(\max_{n_1\le n\le n_2} r_n(\varrho)<\epsilon_0\right) \le \vartheta_3^{n_2 - n_1} |\gamma|.
\]
\end{lemma}

\begin{proof}
We define a descending sequence of ``unlucky'' sub-curves $\tilde{\gamma}_i$ on $\gamma$ as follows.

First, we define  $$\tilde{\gamma}_1=\{\varrho\in\gamma:r_{n_1}(\varrho)<\epsilon_0\}.$$
$\mathcal{F}_0^{n_1}\tilde{\gamma}_1 =\cup_j \gamma_{1,j}$ consists of finitely many pieces and each piece has length $|\gamma_{1,j}|<2\epsilon_0$.

Then for any $\varrho\in\tilde{\gamma}_{1,j}:=\mathcal{F}_0^{-n_1}(\gamma_{1,j})\subseteq\gamma$, we define $k_1(\varrho):=[\log|\gamma_{1,j}|/\log\vartheta_1]+1$. If $n_1+k_1\ge n_2$ on some piece $\tilde{\gamma}_{1,j}$, then we update the definition to $k_1:=n_2 - n_1$ on that piece. We also introduce the time counters $t_1:=n_1$, $t_2(\varrho):=t_1+k_1(\varrho)$ for any $\varrho\in\tilde{\gamma}_{1,j}$.

$k_1$ is a piecewise constant function on $\tilde{\gamma}_1$. Fix $k_1\in\mathbb{N}$ and consider all those pieces $\tilde{\gamma}_{1,j}$ with $k_1(\varrho)=k_1$, then by the second Growth Lemma $$m_{\gamma}(k_1(\varrho)=k_1)\le C_3 \epsilon_0 |\gamma|.$$
Meanwhile by the definition of $k_1$ we have that $|\gamma_{1,j}|<\vartheta_1^{k_1 -1}$, so by the second Growth Lemma $$m_{\gamma}(k_1(\varrho)=k_1)\le C_3 \vartheta_1^{k_1 -1} |\gamma|.$$

We choose $b_1,\vartheta_2\ll 1$ (to be specified later) so small that for all $k_1\in\mathbb{N}$ $$\min\{\vartheta_1^{k_1 -1},\epsilon_0\}\le b_1\vartheta_2^{k_1}.$$
Then $$m_{\gamma}(k_1(\varrho)=k_1)\le C_3 b_1\vartheta_2^{k_1} |\gamma|.$$

We continue the process of going forward by $[\log|\gamma_{1,j}|/\log\vartheta_1]+1$ steps each time and inductively we define a descending sequence of sub-curves $\tilde{\gamma}_{i,j}\subseteq\gamma$ and the functions $k_i,t_i$ on each piece $\tilde{\gamma}_{i,j}$. Also on each piece $\gamma_{i,j}\subseteq \mathcal{F}_0^{t_i}(\gamma)$ we have $$m_{\gamma_{i,j}}(k_{i+1}(\varrho)=k_{i+1})\le C_3 b_1\vartheta_2^{k_{i+1}} |\gamma_{i,j}|$$
Pulling back this estimate to $\gamma$ and considering distortion, on each piece $\tilde{\gamma}_{i,j}$ 
$$m_{\tilde{\gamma}_{i,j}}(k_{i+1}(\varrho)=k_{i+1})\le C_3 K^2 b_1\vartheta_2^{k_{i+1}} |\tilde{\gamma}_{i,j}|$$
where $K$ comes from the distortion control in Proposition \ref{dstr-bd}.

Next we define $$\tilde{\gamma}_{k_1,\cdots,k_i}:=\{\varrho\in\gamma: k_1(\varrho)=k_1,\cdots,k_i(\varrho)=k_i\}.$$
Then by the above estimate 
$$m_{\gamma}(\tilde{\gamma}_{k_1,\cdots,k_i} \cap \{k_{i+1}(\varrho)=k_i+1\}) \le C_3 K^2 b_1 \vartheta_2^{k_{i+1}}.$$

Finally we fix a sequence of natural numbers $k_1,\cdots,k_m$ with $k_1+\cdots+k_m=n_2-n_1$. Then 
\begin{align*}
  &\quad m_{\gamma}(\tilde{\gamma}_{k_1,\cdots,k_i}) \\
  &= |\gamma| \frac{m_{\gamma}(k_1(\varrho)=k_1)}{|\gamma|} \frac{m_{\gamma}(\tilde{\gamma}_{k_1}\cap\{k_2(\varrho)=k_2\})}{|\tilde{\gamma}_{k_1}|} \cdots \frac{m_{\gamma}(\tilde{\gamma}_{k_1,\cdots,k_{m-1}}\cap\{k_m(\varrho)=k_m\})}{|\tilde{\gamma}_{k_1,\cdots,k_{m-1}}|}\\
  &\le |\gamma| C_3 K^2b_1 \vartheta^{k_1} \cdot C_3 K^2b_1 \vartheta^{k_2} \cdot \cdots \cdot C_3 K^2b_1 \vartheta^{k_m}\\
  &\le (C_3 K^2b_1)^m \vartheta_2^{n_2-n_1} |\gamma|
\end{align*}

Now summing over all such possible sequences of natural numbers, we obtain 
\begin{align*}
    m_{\gamma}\left(\max_{n_1\le n\le n_2} r_n(\varrho)<\epsilon_0\right)
    &\le \sum_{m=1}^{n_2-n_1} \binom{n_2-n_1-1}{m-1}(C_3 K^2b_1)^m \vartheta_2^{n_2-n_1} |\gamma|\\
    &\le (1+C_3 K^2b_1)^{n_2-n_1-1} C_3 K^2b_1\vartheta_2^{n_2-n_1} |\gamma|\\
    &\le (\vartheta_2(1+C_3 K^2b_1))^{n_2-n_1} |\gamma|.
\end{align*}

By the definition $\vartheta_3:=\vartheta_2(1+C_3 K^2b_1)$. Now we claim that we can choose $\epsilon_0,b_1,\vartheta_2$ such that $\vartheta_3<1$ and that finishes the proof.

Indeed, we fix some large $k_*\in\mathbb{N}$. We take $\vartheta_2=\vartheta_1^{1/2}$ and $b_1=\vartheta_1^{k/2-1}$. Then for $k_1\ge k_*$ $\vartheta_1^{k_1-1}\le b_1\vartheta_2^{k_1}$ for such choice of $b_1,\vartheta_2$. Next we take $\epsilon_0\le\vartheta_1^{k_*-1}$. Then for all $k_1\in\mathbb{N}$ $$\min\{\vartheta_1^{k_1 -1},\epsilon_0\}\le b_1\vartheta_2^{k_1}.$$
For $k_*$ sufficiently large, $\vartheta_2(1+C_3 K^2b_1)<1$.
\end{proof}

\begin{rmk}\label{main-rmrk}
In fact, we may take $k_*=2$ so that $\vartheta_2=\vartheta_1^{1/2}$, $b_1=1$ and $\epsilon_0\le\vartheta_1$, and Lemma \ref{3growthlemma} holds for such choice parameters as long as we the minimal expansion $\lambda_{\min}\gg1$ is sufficiently large, which can be achieved by choosing $\dot{f}_*$ large. %We also note here that $\vartheta_1=\kappa_1K^2/\lambda_{\min}$ becomes smaller and so does the parameters $\vartheta_2$ as we increase $\dot{f}_*$
\end{rmk}

\begin{definition}[Long curve]\label{longcurve}
An unstable curve $\gamma$ is called \emph{long} if it has size $\vartheta_1/2\leq |\gamma|\leq\vartheta_1$.
\end{definition}

For $\varrho \in \gamma$ we denote by $\bar N(\varrho, \gamma)$ the first time $t=t(\varrho) > 1$ when $\cf_0^tx$ enters a long curve.

\begin{lemma}[Quantitative growth lemma]\label{firsthittime}
Let $\gamma$ be a long curve. Then there exists $b_2>0$ and $\vartheta_4<1$ so that
$$
m_\gamma\Big(\varrho:\bar N(\varrho, \gamma)=N\Big)\leq b_2 \vartheta_4^N |\gamma|.
$$
\end{lemma}

\begin{proof}
We denote by $n_{\gamma}=\left[ \frac{\log|\gamma|}{\log\vartheta_1} \right] +1$. It follows from Lemma \ref{3growthlemma}, by taking $\epsilon_0=|\gamma|/2$, that for $N>n_{\gamma}$ 
\[
    m_{\gamma}\left(\max_{n_{\gamma}\le n\le N} r_n(\varrho)<|\gamma|/2\right) \le \vartheta_3^{N - n_{\gamma}} |\gamma|.
\]
For $N\le n_{\gamma}$, it follows from Lemma \ref{1growthlemma}, by taking $\epsilon=|\gamma|/2$, that 
\begin{align*}
    m_{\gamma}(r_N(\varrho)<|\gamma|/2) 
    &\le \vartheta_1^N m_{\gamma}(r_0(\varrho)<|\gamma|/2) + C_2 |\gamma|/2 |\gamma|\\
    &\le \vartheta_1^N |\gamma| + C_2 |\gamma|/2 |\gamma|\\
    &\le \left( \vartheta_1^N +\frac{C_2|\gamma|}{2} \right) |\gamma|
\end{align*}
Now we choose $\vartheta_4\in(\vartheta_3,1)$ such that $\vartheta_3^{N-n_{\gamma}}\le\vartheta_4^N$ for all $N>n_{\gamma}$. Then we choose $$b_2:=\max_{2\le N\le n_{\gamma}} \left\{\frac{\vartheta_1^N+C_2|\gamma|/2}{\vartheta_4^N} \right\}.$$
Therefore we conclude that for all $N>1$ we have 
\[
   m_\gamma\Big(\varrho:\bar N(\varrho, \gamma)=N\Big)\leq b_2 \vartheta_4^N |\gamma|.
\]
\end{proof}

%\begin{rmk}
%If $|\gamma|=\vartheta_1$, then Lemma \ref{firsthittime} follows directly from Lemma \ref{3growthlemma} with $b_2=1$ and $\vartheta_4=\sqrt{\vartheta_3}$. If $|\gamma|=c\vartheta_1$ for some constant $c$ of order 1, then $n_{\gamma}=2$ and $$b_2=\max \left\{\frac{\vartheta_1+C_2|\gamma|/2}{\vartheta_4}, \frac{\vartheta_1^2+C_2|\gamma|/2}{\vartheta_4^2}\right\}.$$
%\end{rmk}

\subsection{A growth Lemma with a delay}\label{delay}
We now let the dynamics run for $N_0$ many times and then start measuring long/short curves.
Let $\gamma$ be a long curve. 
Fix some $N_{0}\geq 1$ and define a map $\widehat{N}(\varrho): \gamma \rightarrow \mathbb{N}$
as follows: for each $\varrho \in \gamma$,  let $\widehat{N}(\varrho)=N_{0}+ \bar{N}\left(\cf_0^{N_0}\varrho, \gamma'\right)$, where
$\gamma'$ is the unstable curve that contains $\cf_0^{N_0}\varrho$. 
We have the following Lemma:

\begin{lemma}\label{pr-delayed}
For all $\ell>N_0$ and every long curve  $\gamma$ (see Def. \ref{longcurve}) we have that
$$
m_\gamma(\varrho: \hat N(\varrho)=\ell)\leq b \vartheta_4^{\ell-N_0}|\gamma|,
$$
Moreover, for $\f$ large enough, the constants $b,\vartheta_4$ are independent of $\dot f_1,\dot f_2$, with $|\dot f_1|,|\dot f_2|\geq \f$.
\end{lemma}
\begin{proof}
Note that the set $\{\varrho:\hat N(\varrho)=\ell\}$ consists of two types of points. First, the ones that never visit a long curve throughout their journey up to time $\ell$. The measure of this set can be estimated by the third growth Lemma \ref{3growthlemma} as $b_2 \vartheta^\ell$.
For $\varrho \in \gamma$ let $\{\gamma_\varrho\}$ by the last long curve the trajectory of $\varrho$ visits before time $N_0$. From $\gamma_\varrho$ to $\ell$ the particle will have to make at least $\ell-{N_0}$ many steps. Applying the third growth lemma to $\gamma_\varrho$ we see that the measure of the points that will arrive at a long curve after $\ell-N_0$ is at most $Kb \vartheta^{\ell-N_0}$, where $K$ is the distortion. Thus
$$
m_\gamma(\varrho: \hat N(\varrho)=\ell)\leq  b_2\vartheta_4^\ell+K b_2\vartheta_4^{\ell-N_0}\leq 3b_2\vartheta_4^{\ell-N_0}.
$$
By taking $b=3b_2$, we get the desired estimate. As it was pointed out in Remark \ref{main-rmrk}, the constants $b_2$ and $\vartheta_4$ will only decrease as we increase $\f$. Hence, we can choose $b, \vartheta_4$ to be the uniform upper bounds of all pairs $(b, \vartheta_4)$ for all  $\dot f_1,\dot f_2$, with $|\dot f_1|,|\dot f_2|\geq \f$. This will prove the last statement of the lemma.
\end{proof}

\subsection{A large deviation bound}\label{dv-est}

We now iterate the long curves obtained in the previous Lemma. Namely, define $\hat N_k$ inductively. We set $N_1=\hat N$ and let $\{\gamma_\varrho\}$ be the collection of all long curves in $P_0^{\hat N}\gamma$. Suppose $P_0^{\hat N}\varrho\in \gamma_0$, where $\gamma_0$ is a long curve. Then set $\hat N_k(\varrho)=\hat N_{k-1}(P_0^{\hat N_{k-1}}\varrho, \gamma_0)$. 

\begin{lemma}\label{dev-prop}
There exists $a=a(N_0)>0$ and $\vartheta_5<1$ such that
$$
m_\gamma(\varrho:\hat N_n(\varrho)> an)\leq \vartheta_5^n |\gamma|.
$$
\end{lemma}
\begin{proof}
First note that for any $c>0$ we can write
$$
\E_\gamma[e^{c \hat N}]\leq \sum_{k=N_0}^\infty 3 \vartheta_4^{k-N_0} e^{ck}< 3\vartheta_4^{-N_0} \sum_{k=N_0}^\infty (\vartheta_4 e^{c})^k.
$$
If $c$ is small enough then $\vartheta_4 e^{c}<1$. Hence
$$
\E[e^{c \hat N}]\leq Cb\vartheta_4^{-N_0} (\vartheta_4 e^{c})^{N_0}=\rho<\infty.
$$
Let $\mathcal{A}_k$ be the $\sigma$-algebra generated by the partition of the long curve at step $k$ by the intervals of constancy of $\hat N_k$. Then
$$
\E[e^{c\hat N_n}]=\E[\E[e^{c\hat N_n}|\mathcal{A}_{n-1}]]=\E[e^{c\hat N_{n-1}}\E[e^{c(\hat N_n - \hat N_{n-1})}|\mathcal{A}_{n-1}]]
$$
$$
\leq \E[e^{c\hat N_{n-1}}]\rho \leq \rho^n.
$$
Thus for each $\lambda>0$
$$
e^{\lambda} m_\gamma(\varrho: e^{c\hat N_n(\varrho)}>e^{\lambda})\leq \rho^n |\gamma|.
$$
Choose $\lambda=nq$. Then
$$
m_\gamma(\varrho: e^{c\hat N_n(\varrho)}>e^{nq})=m_\gamma(\varrho: c{\hat N_n}(\varrho)>nq)\leq \Big(\frac{\rho}{e^q}\Big)^n.
$$
Taking $q$ so large that $e^q>\rho$ and setting $\vartheta_4=\rho/e^q$ and $a=q/c$ we get that
$$
m_\gamma(\varrho: \hat N_n(\varrho)>na)\leq \vartheta_4^n |\gamma|.
$$
\end{proof}

%-------------------------------------------------------------

\section{Energy Growth in Finite Time}\label{fnt-time}

In this section we show that for sufficiently large integer $N_0$ exponential energy growth can be achieved for every long curve $\gamma$ for a substantial portion of initial conditions $\varrho=( x_0,z_0) \in \gamma$ (Proposition \ref{N0}), if the parameters $\V$ and $\f$ are chosen to be sufficiently large and $f_1 \neq f_2$. 

\begin{comment}

Throughout this section, for the ease of presentation $P$ will denote the modified system (denoted by $P_0$) and respectively $\cf=P^2$. It will also be assumed that the unstable curves $\gamma$ considered in this Section, lay above the level $\V$. 
We start by a definition:

\end{comment}

\begin{definition}[Complete curves]\label{D1}
An unstable curve $\gamma$ will be called complete if $\gamma$ runs across the entire $\partial F_{P_0,m}$ from the top to the bottom, for some $P_0 \in \mathcal{A}$ and index $m$ (see Def. \ref{complete-crv}).
\end{definition}

\begin{figure}[!ht]
    \centering
    \begin{tikzpicture}
       \draw[->] (-0.3,0) -- (2.5,0);
       \draw[->] (0,0) -- (0,3);
       \draw (0,0) node[anchor=north]{0}
             (2,0) node[anchor=north]{2};
       \draw (2,0)--(2,2.8);
       \draw (0,0.8)--(2,1.5) (0,1.4)--(2,2.1);
       \draw[thick,red] (0.7,1.65)--(0.9,1.1);
       
       \draw[->] (3.7,0) -- (6.5,0);
       \draw[->] (4,0) -- (4,3);
       \draw (4,0) node[anchor=north]{0}
             (6,0) node[anchor=north]{2};
       \draw (6,0)--(6,2.8);
       \draw (4,0.8)--(6,1.5) (4,1.4)--(6,2.1);
       \draw[thick,blue] (4.8,1.5)--(4.9,1.2);
    \end{tikzpicture}
    \caption{Complete and incomplete curves}
    \label{fig:completecurves}
\end{figure}
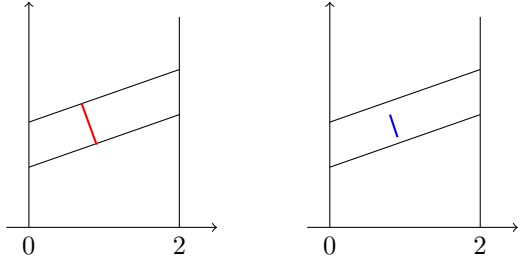

The next Lemma shows that every unstable curve can essentially be identified with its projection to the vertical axes, up to a distortion constant close to $1$.
\begin{lemma}\label{vrs}
 Let $\gamma$ be an unstable curve. There exists a constant $1<K=K(\f)$ such that if $(x_1, y_1)$ and $(x_2, y_2)$ are the endpoints of $\gamma$, then
\begin{equation}\label{proj}
1\leq \frac{|\gamma|}{|y_1-y_2|}\leq K,
\end{equation}
and $\lim_{|\f| \rightarrow \infty}K=1$. Next, if $\gamma$ is in $R_1$, with $|\gamma|\gg 2$, then
\begin{equation}\label{U11}
\frac{|\gamma \cap S_{+1}|}{|\gamma|}=f_2+o(|\gamma|)
\end{equation}
and respectively
\begin{equation}\label{L11}
\frac{|\gamma \cap S_{-1}|}{|\gamma|}=(1-f_2)+o(|\gamma|).
\end{equation}
\end{lemma}
\begin{proof}
\eqref{proj} follows from the fact that $\gamma \in C^2$ and the cones $\mathcal{C}_{u}^*$ are almost vertical.

For \eqref{L11} and \eqref{U11} we note 
that for every complete curve $\gamma_1 \subset S_{+1}$ and $\gamma_2 \subset S_{-1}$ one has that $\frac{f_2}{1-f_2}K\leq |\gamma_1|/|\gamma_2|\leq K \frac{f_2}{1-f_2}$ and for every complete curve $\gamma_1 \subset S_1$ and $\gamma_2 \subset S_{1}$ one has that $1/K\leq |\gamma_1|/|\gamma_2|\leq K$. And since the connected components
$\{\mathcal{L}_{en,m}^l\}_{m},\{\mathcal{L}_{en,m}^s\}_{m}$ and $\{\mathcal{U}_{u,m}\}_{m}$ interchange each other (c.f. Figure \ref{fig:proportionlongcurve}), then the estimates \eqref{L11} and \eqref{U11} follow. The proportion of non-complete curves at the endpoints of $\gamma$ is of size $o(|\gamma|)$.
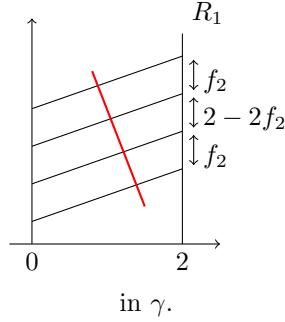
\begin{figure}[!ht]
    \centering
    \begin{tikzpicture}
       \draw[->] (-0.3,0) -- (2.5,0);
       \draw[->] (0,0) -- (0,3);
       \draw (0,0) node[anchor=north]{0}
             (2,0) node[anchor=north]{2};
       \draw (2,0) -- (2,2.8) node[anchor=south west]{$R_1$};
       \draw (0,0.3) -- (2,1)
             (0,0.8) -- (2,1.5)
             (0,1.3) -- (2,2)
             (0,1.8) -- (2,2.5);
       \draw[<->] (2.15,1.05) -- (2.15,1.45) node[anchor=north west]{$f_2$};
       \draw[<->] (2.15,1.55) -- (2.15,1.95) node[anchor=north west]{$2-2f_2$};
       \draw[<->] (2.15,2.05) -- (2.15,2.45) node[anchor=north west]{$f_2$};
       \draw[thick,red] (0.8,2.3) -- (1.5,0.5);
    \end{tikzpicture}
    \caption{Proportions of $S_{\pm1}$} in $\gamma$. 
    \label{fig:proportionlongcurve}
\end{figure}

\end{proof}

Denote $q_1=f_2$ and $q_{-1}=1-f_2$. For $t\in \{1,-1\}$ and an unstable curve $\gamma$ let
$$
A_{\gamma, t}=\{z_0\in \gamma: z_1\in S_{t}\},
$$
and denote by $\Gamma_{\gamma, t}$ the collection of all complete curves in $\cf_0\gamma \cap S_t$. 

We have the following Lemma.

\begin{lemma}\label{L1}
For every $\varepsilon>0$, $\f$ can be taken so large that if $|\dot{f}_1|,|\dot{f}_2|>\dot{f}_*$, then for every $t \in \{-1,1\}$ and any complete curve $\gamma$ in the $R_1$ cylinder we have that
\begin{equation}\label{Q1}
m_\gamma\Big(A_{\gamma, t}\Big)\leq (q_{t}+\varepsilon)|\gamma|,
\end{equation}
and for the collection $\Gamma_{\gamma,t}$ of all complete curves in $\cf_0\gamma \cap S_{t}$ we have that 
\begin{equation}\label{Q2}
m_\gamma\Big(A_{\gamma, t}\setminus \Big(\bigcup_{\ell \in \Gamma_{\gamma,t}}P^{-1}\ell\Big)\Big)\leq \varepsilon|\gamma|.
\end{equation}
\end{lemma}
\begin{proof}

We consider the case $\gamma \subset S_{+1}$, i.e. the energy travels through the upper route. The case $\gamma \subset S_{-1}$ is analogous.
Note that the curve $P_0\gamma$ will first be mapped into $R_2$ (see Figure \ref{fig:dividedphase}). It will be of size close to $|\gamma|\f$. In $R_2$ it will be cut into multiple components by singularity lines. Note that there will only be two non-complete curves at the endpoints of $P_0\gamma$.
Clearly, the proportion of the two non-complete curves compared with the complete ones in between will be negligible for large $\f$. We next consider the complete curves $\{\gamma'\}$ on $P_0 \gamma$ in $R_2$ and iterate them further into the $R_1$ cylinder. Note that $P_0\gamma'$ will be an unstable curve of size $\f|\gamma'|$, where $|\gamma'| \approx 1/2$ and its endpoints will be at the vertical lines $\{\sigma=f_1\}$ and $\{\sigma=2-f_1\}$.
Obviously, by taking $\f$ large the proportion of non-complete curves in $P_0 \gamma'$ can be made arbitrarily small. This proves \eqref{Q2} since whatever is left out of the complete curves can be made to have measure less than $\varepsilon |\gamma|$. 
To see \eqref{Q1} we note that by Proposition \ref{vrs} the proportion of the sets $P_0 \gamma' \cap S_1$ and $P_0 \gamma' \cap S_{-1}$ in the $R_1$ cylinder are respectively of size $q_1$ and $q_{-1}$. Then, the distortion Lemma \ref{dstr-bd} lets us to conclude that the pre-images of these sets in $\gamma'$ and then in $\gamma$ is again going to be of size $q_1$ and $q_{-1}$. This completes the proof.

\end{proof}

We now extend the definition of $A_{\gamma,t}$ given above. For an unstable curve $\gamma$ and an itinerary $t=(t_1, \dots, t_{n})$, with $t_k = \pm 1$, $k \leq n$, and $1 \leq m \leq n$ denote
\begin{equation}
A_{\gamma, t,m}=\{
z_0 \in \gamma: z_1 \in S_{t_1}, z_2 \in S_{t_2} \\
\dots z_{m} \in S_{t_m}
\}.
\end{equation}

We have the following extension of the previous Lemma.
\begin{lemma}\label{ind-trj}
For every $n\geq 1$, $\varepsilon>0$, $\f$ can be taken so large that, if $|\dot{f}_1|,|\dot{f}_2|>\dot{f}_*$,
then for every complete curve $\gamma$ in the $R_1$ cylinder and any itinerary $t=(t_1, \dots, t_n)$ we have that
$$
m_\gamma(A_{\gamma,t,n})\leq \Big(q_{t_1}q_{t_2}\dots q_{t_n} + \varepsilon\Big)|\gamma|.
$$
\end{lemma}
\begin{proof}
The idea is to iterated the complete curves obtained in the previous Lemma. We wish to show that for every $\varepsilon>0$ and every $m$, with $1\leq m \leq n$, one can take $\f$ so large that
\begin{equation}\label{cmpl1}
m_\gamma(A_{\gamma, t,m})/|\gamma|\leq \Big(\prod_{k \leq m} q_{t_k}\Big) + \varepsilon,
\end{equation}
and there exist a collection of complete curves $\Gamma_m$ in $\cf^{m}\gamma$ so that
\begin{equation}\label{cmpl}
m_\gamma(A_{\gamma, t,m}\setminus \bigcup_{\ell \in \Gamma_{k}}\cf^{-m}\ell)|\leq \varepsilon |\gamma|.
\end{equation}
Note that the first step $m=1$ follows from Lemma  \ref{L1}. Assume we have already proven the statement for $m < k$. For $m=k$ consider the collection $\Gamma_{k-1}$ constructed at step $k-1$ and for each $\ell \in \Gamma_{k-1}$ consider the set $A_{\ell, t_{k}}$ and the associated collections $\Gamma_{\ell,t_k}$. Define $\Gamma_k=\cup_{\ell \in \Gamma_{k-1}}\Gamma_{\ell, t_k}$. By Lemma \ref{L1} the proportion of the set $A_{\ell, t_{k}}$ inside $\ell$ will be not greater than $q_{t_m}+\varepsilon$. Hence
$$
m_\gamma(A_{\gamma,t,k})\leq m_\gamma(A_{\gamma, t,k-1}\setminus \bigcup_{\ell \in \Gamma_{k-1}} \cf^{k-1}\ell)+
 \Big(\prod_{s=1}^{k-1} q_s + \varepsilon\Big)(q_k+ \varepsilon) K|\gamma|,
$$
where the factor $K$ is to account for the distortion. The first term is small due to \eqref{cmpl} while the second term can be made arbitrarily close to $\Big(\prod_{s=1}^{k} q_s \Big)|\gamma|$. As for the collection $\Gamma_k$, note that
$$
m_\gamma\Big(A_{t,k}\setminus\bigcup_{\ell \in \Gamma_k}\cf^{-k}\ell\Big)\leq \varepsilon |\gamma|+\Big(\prod_{s=1}^{k-1} q_k + \varepsilon\Big)K\varepsilon |\gamma|,
$$
where $\varepsilon$ is the contribution of points outside of the set $A_{t,k-1}\setminus \bigcup_{\ell \in \Gamma_{k-1}}\ell$, while the second term is the contribution of the set outside of $A_{\ell, t_k}\setminus\Gamma_{\ell, t_k}$. It is now clear that if we had started the procedure with $\varepsilon$ small enough then the estimate for $m=k$ for $\varepsilon$ would follow. 
Repeating the procedure until $m=n$ and taking $\varepsilon$ small enough, we will arrive at the result.

\end{proof}

In the above Lemma we were able to estimate the probability of the energy of the ball of following a given pattern $t$. Note that we also have a lower bound on the energy change of the particle in Proposition \ref{M0} (3). We now combine these two information in order to estimate the probability of the energy of the ball of beeing too low after $N_0$ many step. The idea is to approximate the dynamics by a markov chain.

\begin{lemma}\label{main-subsect}
Assume $f_1>f_2$ and let
$$
\mathcal E=(1-f_2)\log \frac{1-f_2}{1-f_1}+f_2 \log \frac{f_2}{f_1}.
$$
Let $a>0$ and $N_0$ be a sufficiently large integer. Then $\f$ and $\V$ can be taken so large that if $|\dot f_1|,|\dot f_2|>\f$, then for the modified dynamical system $P_0$ and every unstable curve $\gamma$, with $|\gamma|>a$ we will have that
$$
\mathbf{E}_\gamma\Big[\frac{\ln z_{N_0} - \ln z_0}{N_0}\Big]\geq \frac{\mathcal E}{2}.
$$
\end{lemma}
\begin{proof}
Let $X, Y$ be discrete measures defined on the set $\{0,1\}$ so that $X(0)=1-f_2$ and $X(1)=f_2$ and $Y(0)=1-f_1$ and $Y(1)=f_1$. Note that
\[
   D(X \parallel Y) = (1-f_2)\log\frac{1-f_2}{1-f_1} + f_2 \log\frac{f_2}{f_1},
\]
where $D(\cdot \parallel \cdot)$ is the Kullback-Leibler divergence. It is well known that
$D(X \parallel Y)\geq 0$ and $D(X \parallel Y)=0$ if and only if $X=Y$ or equivalently $f_1=f_2$. Hence, if $f_1 \neq f_2$, then $\mathcal{E}>0$.

Note that since $|\gamma|>a$, then $\cf_0 \gamma$ will be of size at least $|\f|^2a$. Hence, if $\f$ is sufficiently large then the proportion of non-complete curves in $\cf_0 \gamma$ can be made arbitrarily small as discussed in the proof of Lemma \ref{L1}. So it is sufficient to prove the Lemma for the complete curves in $\cf_0\gamma$. Or, we can assume that $\gamma$ is a complete curve in $R_1$.

By Proposition \ref{M0} we have for $\varrho = (x_0, z_0) \in S_{+1}$ that
$$
\ell\left(\frac{f_2}{f_1}\right) z_0 \leq z_1,
$$
and respectively for the upper route $ \varrho = (x_0, z_0)\in S_{-1}$
$$
\ell\left(\frac{1-f_2}{1-f_1}\right) z_0\leq z_1.
$$
For $\lambda_0>0$, let
$$
\lambda_{n+1} = \lambda_{n} d_n,
$$
where $d_n$ is an i.i.d. sequence of discrete random variables so that  $\mathbb{P}(d_n=\ell\frac{f_2}{f_1})=1-f_2$, $\mathbb{P}(d_n=\ell(1-f_2)/(1-f_1))=1-f_2$. By assumption $\mathcal E>0$. Hence, if $\ell$ is sufficiently close to $1$ then
$$
E[\ln d_n]=(1-f_2)\log \ell\left(\frac{1-f_2}{1-f_1}\right)+f_2 \log \ell\left(\frac{f_2}{f_1}\right)=\mathcal E + \log \ell>0.
$$
Let $h=\mathcal E + \log \ell>0$. Then by the Hoeffding's inequality \cite{WH1963}, for every $0\leq r <h$ we have that
$$
\Pr\left(\sum_{k=0}^{n-1} \ln d_k < (h-r)n\right)\leq e^{-2n r^2}.
$$
Equivalently
\begin{equation}\label{estimate}
\Pr\left(\lambda_n < e^{(h-r)n}\lambda_0\right)\leq e^{-2n r^2}.
\end{equation}
Let $t=(t_1, \dots, t_{N_0})$, $t_k \in \{0,1\}$ and $t'=(t_1', t_2', \dots, t_{N_0}')$, with $t_k'=2(t_k-1/2)$, for all $k\leq N_0$. Define also $t''=(t_2, \dots, t_{N_0})$.
Let $\gamma \subset S_{t_1}$. Note that for $z_0 \in A_{\gamma, t'',N_0-1}$ we obtain by iterating the relations above (observe that in the definition of $A_{\gamma, t'',N_0-1}$ we skip the first step and because of that we needed to separate $t_1$)
$$
 \prod_{k=1}^{N_0}\left( \ell\left(\frac{f_2}{f_1}\right)\right)^{t_k}\left(\ell\left(\frac{1-f_2}{1-f_1}\right)\right)^{1-t_k}z_0 \leq z_{N_0}.
$$
We also have a bound on the measure of the set $A_{\gamma, t',N_0}$ by Lemma \ref{ind-trj}
$$
m_\gamma(A_{\gamma, t'',N_0-1})\leq \Big(\prod_{k=2}^{N_0}q_{t_k'} + \varepsilon\Big)|\gamma|,
$$
where $\varepsilon>0$ is arbitrary.
If we denote by $\lambda_n(t')$ the probability that the random walk $\lambda_n$ follows the pattern $t$, then clearly
$$
\lambda_{N_0}(t)=\prod_{k=1}^{N_0}q_{t_k'}
$$
and
$$
\lambda_{N_0}|_t = \prod_{k=1}^{N_0}\left( \ell\left(\frac{f_2}{f_1}\right)\right)^{t_k}\left(\ell\left(\frac{1-f_2}{1-f_1}\right)\right)^{1-t_k}\lambda_0.
$$
Set $\lambda_0= \min\{z_0:(x_0,z_0) \in \gamma\}\}$. Since $\lambda_0\leq z_0$, for every $\varrho \in \gamma$, then $\lambda_{N_0}|_t \leq z_{N_0}|_t$. Now let $\omega_{N_0}$ be the collection of all itineraries $t''$ of length $N_0-1$ for which there exists $\varrho \in A_{\gamma,t'', N_0-1}$ so that $z_{N_0}<e^{N_0(h-r)} z_0$. Then 
$$
m_\gamma(\varrho\in \gamma: z_{N_0}< e^{(h-r)N_0}z_0)/|\gamma|\leq \sum_{t'' \in \omega_{N_0}}|A_{\gamma, t'',N_0-1}|\leq
$$
$$
\leq \sum_{t'' \in \omega_{N_0}} \Big(\prod_{k=2}^{N_0}q_{t_k'}+\varepsilon \Big)\leq \frac{1}{q_{t_1'}}\left(1+\frac{\varepsilon}{q^{N_0-1}}\right) \sum_{t''\in \omega_{N_0}}\prod_{k=1}^n q_{t_k'}
$$
$$
\leq \frac{1}{q_{t_1'}}\left(1+\frac{\varepsilon}{q^{N_0-1}}\right) \Pr(\lambda_{{N_0}}< e^{(h-r)N_0}z_0),
$$
where $q=\min\{q_1, q_{-1}\}$. The estimate in the last line was due to the fact that $\lambda_{N_0}|_t \leq z_{N_0}|_t$ for any itinerary $t$.

Since $\gamma$ is a complete curve of size less than $1$, then for every $\varrho=(x_0, z_0) \in \gamma$ we will have $z_0\leq 2\lambda_0$, . Take $r'=r/\sqrt{2}$. Then for $N_0$ large
$$
\mathbb{P}(\lambda_{N_0}<e^{(h-r)N_0}z_0)\leq \mathbb{P}(\lambda_{N_0}<e^{(h-r)N_0}2\lambda_0)
$$
$$
\leq  \mathbb{P}(z_{N_0}<e^{(h-r')N_0}z_0)\leq e^{-2N_0r'^2}=e^{-N_0r^2}.
$$
Since $\varepsilon$ was arbitrary, then we can take it so small (at the expense of taking $\f$ and $\V$ large) that
$$
m_\gamma(\varrho\in \gamma: z_{N_0}< e^{(h-r)N_0}z_0)\leq Ce^{-N_0 r^2}|\gamma|.
$$
Since $h=\mathcal{E}+\log \ell$, then for $r$ small enough and $\ell$ close to one, we will have for some constant $C>0$
$$
m_\gamma(\varrho\in \gamma: z_{N_0}< e^{(2\mathcal{E}/3)N_0}z_0)\leq Ce^{-N_0 r^2}|\gamma|.
$$
Denote
$$
A_{N_0}= \left\{\varrho\in \gamma:\frac{\ln z_{N_0}-\ln z_{0}}{N_0}<\frac{2\mathcal{ E}}{3}\right\}
$$
Then $m_\gamma(A_{N_0})<Ce^{-N_0 r^2}|\gamma|$. Hence
$$
\mathbf{E}_\gamma\Big[\frac{\ln z_{N_0} - \ln z_0}{N_0}\Big]\geq \Big(1-|A_{N_0}|\Big)\frac{2\mathcal E}{3}-|A_{N_0}|\frac{2\mathcal E}{3}\geq \frac{\mathcal E}{2},
$$
for sufficiently large $N_0$.

\end{proof}

We now show that above lemma is correct if $\gamma$ is a long curve:

\begin{prop}\label{N0}
Assume that $f_2> f_1$. Then for all sufficiently large $N_0$ one can choose $\f$ and $\V$ so large that for the modified dynamical system $P_0$ defined in Proposition  \ref{M0}, we have for every long curve $\gamma$ on the $R_1$ cylinder that
$$
\mathbf{E}_\gamma\Big[\frac{\ln z_{N_0}(\varrho) - \ln z_0}{N_0}\Big]\geq \frac{\mathcal{E}}{3},
$$
where
$$
\mathcal{E}=(1-f_2)\log\frac{1-f_2}{1-f_1} + f_2 \log\frac{f_2}{f_1}.
$$
\end{prop}
\begin{proof}
Recall that by the discussion at the beginning of the previous Lemma we have that $f_1 \neq f_2$ implies $\mathcal{E}>0$.

According to the discussion at the beginning of Subsection \ref{grwth} the image of  every curve $\gamma$ of size $\vartheta_1$ will be cut into at most 4 pieces under the map $\cf_0 \gamma$. Note that $|\mathcal{F}_0\gamma|$ will also have length uniformly bounded from below by a number that is independent of $\Lambda_1,\Lambda_2$, due to Remark \ref{lbd-f}. Let $\Gamma=\{\gamma_k\}_{k=1}^s$, $s\leq 4$ be the curves so that their lengths are ordered in a decreasing order and
$$
\sum_{s}|\gamma_s|=|\cf_0 \gamma|=b.
$$
Take some $\rho<1$. We now choose a sub-collection $\Gamma_0$ of curves from $\Gamma$. We have that $|\gamma_1| \geq b/4$. If $|\gamma_1|>b(1-\rho)$ we define $\Gamma_0=\{\gamma_1\}$ and stop. Otherwise, if $|\gamma_1|<b(1-\rho)$, then note that
$$
\sum_{s>1}|\gamma_s| \geq b - b(1-\rho)=b \rho.
$$
Then $|\gamma_2| \geq b\rho/3>b\rho/4$. If now $|\gamma_1| + |\gamma_2| > b(1-\rho)$, then we stop and set $\Gamma_0=\{\gamma_s\}_{s=1}^2$. Otherwise we will either have $|\gamma_3|>b\rho/4$ and subsequently $|\gamma_4|>b\rho/4$ or we can find a sub-collection $\Gamma_0$ so that $\sum_{\gamma_0 \in \Gamma_0}|\gamma_0|>b(1-\rho)$ and $|\gamma_0|>b\rho/4$ for every $\gamma_0 \in \Gamma_0$.

Since for $\gamma_0 \in \Gamma_0$ we have that $|\gamma_0|>\rho|\cf_0 \gamma|/4$ then for fixed $\rho$, by Lemma \ref{main-subsect}, we can choose $N_0$ and $\f$ in such a way that
$$
\mathbf{E}_{\gamma_0}\Big[\frac{\ln z_{N_0} - \ln z_0}{N_0}\Big]\geq \frac{\mathcal E}{2}.
$$
Next, by the choice of $\rho$ and in view of \eqref{bound} 
$$
\mathbf{E}_{\gamma \setminus \cup_{\gamma_0 \in \Gamma_1}\gamma_0}\Big[\left|\frac{\ln z_{N_0} - \ln z_1}{N_0}\right|\Big]\leq K\rho c,
$$
Hence, by choosing $\rho$ small enough (and respectively increasing the values of $\f$ and $N_0$) we will have
$$
\mathbf{E}_\gamma\Big[\frac{\ln z_{N_0} - \ln z_1}{N_0} \Big]\geq \frac{1}{K}\frac{\mathcal E}{2}(1-\rho)- K\rho c  \geq \frac{\mathcal E}{3}.
$$

\end{proof}

%-------------------------------------------------------------

\section{Exponential Acceleration}

In the previous section we showed that for any sufficiently large $N_0$ and for an appropriate choice of $\f$ and $\V$, energy growth can be achieved for a substantial portion of initial conditions $\varrho\in \gamma$ uniformly for every long curve $\gamma$ above $\V$. After $N_0$ many iterations $\cf_0^{N_0}\gamma$ will consist of multiple unstable curves some of which will be long and some will be short. If we had only long curves, then after another $N_0$ many iterations ($\cf_0^{2 N_0}$) each of the long curves would produce some further exponential acceleration. Continuing like this we would get acceleration at each step for a large set of initial conditions. Understandably, this cannot continue forever and there will eventually emerge some short curves which may not produce acceleration after $N_0$ may iterations. In order to overcome this problem, we will need to iterate these curves until they become long. The tool for keeping track of the waiting times in this process will of course be the growth lemmas proven in Section 6. One can recognize the  procedure described above as the first step $\hat N_1(\varrho)$ defined in Section \ref{dv-est}.
After this, we will repeat the procedure and subsequently define $\hat N_2(\varrho), \hat N_3(\varrho)$ etc.

Of course, when $\cf_0^n \varrho$ is in a short curve and it is taking the curve a long time to grow, the energy of the ball may change uncontrollably, start to decrease and drop even below the level it had started with. As a result of this the entire energy gain up until that time will be lost. However, we will show that this won't happen too often and the exponential energy growth obtained in finite intervals of length $N_0$ will eventually persist in infinite time for a set of initial conditions of large measure.

To handle the iteration process described above, we prove a moment estimate.
Let $\Delta(\varrho)=\ln z_{\hat{N}(\varrho)}-\ln z_0$, where $\hat N$ is defined before Lemma \ref{pr-delayed}.

\begin{prop}\label{Lm1}
In Prop. \ref{N0} the parameters $N_0$, $\f$ and $\V$  can be chosen in such a way that there will exist constants $\kappa>0, \vartheta<1$ depending on $f_1, f_2, \mathcal E$ so that for the modified map $P_0$, defined in Prop. \ref{M0}, and for every long curve $\gamma$ in  $R_1$ we have that
$$
\E_\gamma[e^{-\kappa \Delta}]\leq \vartheta.
$$
\end{prop}
\begin{proof}
Let $\kappa=\eta/N_0$ where $\eta$ is such that for $|s| \leq 2\eta$ we have
$$
e^{-s}\leq 1 - s + s^2.
$$
Then
$$
\E_\gamma[e^{-\kappa \Delta}]=\E_\gamma[e^{-\kappa \Delta}1_{\hat N\leq 2N_0}]+\E_\gamma[e^{-\kappa \Delta}1_{\hat N> 2N_0}]
$$
$$
\leq 1 - \E_\gamma[\kappa \Delta 1_{\hat N \leq 2N_0}]+\E_\gamma[(\kappa \Delta)^2 1_{\hat N \leq 2N_0}]+\E_\gamma[e^{-\kappa \Delta} 1_{\hat N>2N_0}].
$$
Since by Proposition \ref{M0} we have that $|\kappa \Delta|=\kappa |\ln z_{\hat{N}(x)}-\ln z_0|\leq \kappa c\hat N $. Then
$$
\E_\gamma[e^{-\kappa \Delta}]\leq 1 - \E_\gamma[\kappa \Delta]+4\eta^2 + \E_\gamma[(e^{-\kappa \Delta}+\kappa \Delta) 1_{\hat N>2N_0}].
$$
Next
$$
\E_\gamma[\Delta]=\E_\gamma\left[\ln z_{N_{0}}-\ln z_{0}\right]+\E_\gamma\left[\ln z_{\hat N}-{\ln x}_{N_{0}}\right] 
$$
For the second term by Lemma \ref{pr-delayed}
$$
\E_\gamma\left[|\ln z_{\hat N}-\ln {z}_{N_{0}}|\right]\leq Cb\sum_{m=N_0}^\infty (m-N_0) \vartheta_4^{m-N_0} <\infty.
$$
For the first term, by Proposition \ref{N0} we have that
$$
\E_\gamma\left[\ln z_{N_{0}}-\ln z_{0}\right]\geq \frac{\mathcal E}{3} N_{0}.
$$

Thus
$$
\E_\gamma[\Delta]=\frac{\mathcal EN_0}{2}+O(1)=\Big(\frac{\mathcal E}{3} +o(1)\Big)N_0.
$$
Next note that
$$
\E_\gamma[(e^{-\kappa \Delta}+\kappa \Delta) 1_{ \hat{N}\geq 2N_0}]\leq 2\E[e^{\kappa |\Delta|} 1_{ \hat{N}\geq 2N_0}]
$$
\begin{align}
&\leqslant C_0b \sum_{m \geq 2 N_{0}} e^{c\kappa\left(m-N_{0}\right)} \vartheta_4^{\left(m-N_{0}\right)}=C_0b \sum_{m \geq N_{0}} (e^{c\kappa } \vartheta_4)^{m}
\end{align}
We now take $N_0$ so large that $e^{c\kappa}$ is small, with $e^{c\kappa}\vartheta_4<1$. Note that $N_0$ can be increased without affecting the values of $c,b, \vartheta_4$ (however not the size of $\gamma$). This is done as follows. We first take $\f$ large so that we have unstable cone invariance and that the constants in Lemma \ref{pr-delayed} are uniform for all $|\dot f_1|, |\dot f_2|>\f$. Thus, the constants $c,b, \vartheta_4$ will be fixed for all large $\f$. Thus, to achieve Proposition \ref{N0}, it will remain to further increase the values $\f$ and $\V$.

We will have from above that
$$
\E_\gamma[e^{-\kappa \Delta}+\kappa \Delta) 1_{ \hat{N}\geq 2N_0}] \leq C^{*}\left(e^{c\kappa} \vartheta_4\right)^{N_{0}}.
$$
Summarizing
$$
\E_\gamma[e^{-\kappa \Delta}]\leq 1 -  \Big(\frac{\mathcal E}{2} +o(1)\Big)\eta  + 4\eta^2 + C^{*}e^{c\eta} \vartheta_4^{N_{0}}.
$$
We remark that the parameters $\kappa, \eta$ are independent and they can be made arbitrarily small at the expense of taking $N_0$ large.
We first take $\eta$ small so that $- \Big(\frac{\mathbf{v}}{2} +o(1)\Big)\eta+4\eta^2$ is negative. Then, we keep $\eta$ fixed and increasing $N_0$, so that $e^{c\eta} \vartheta_4^{N_0}$ becomes small compared with the remaining terms. As a result $\E_\gamma[e^{-\kappa \Delta}]$ will become less than $1$. 

We remark that the constants $\kappa, \vartheta$ and $c$ depend on the properties of the modified system $P_0$ and the constants in the growth Lemma \ref{pr-delayed}. The properties of $P_0$, in turn, depend on $f_1, f_2, \mathcal{E}$, but not on the parameters $\dot f_1, \dot f_2$. Recall, however, that this due to the choice of $\V$ and $\V=\V(\dot f_1, \dot f_2, f_1, f_2)$.
The constants in the growth lemmas depend only on $f_1, f_2$, for large $\f$, as they measure the distance between singularity lines and hence contribute into the complexity estimates.
The quantity $\mathcal{E}$ measures how close $f_1$ and $f_2$ are and it contributes into the rate of the drift. Thus, we conclude that for sufficiently large $\f$ and $\V$, the constants $\kappa$ and $\vartheta$ depend on $f_1, f_2$ and $\mathcal{E}$, but not on $\dot f_1$ and $\dot f_2$.
This completes the proof.
\end{proof}

For $n \geq 1$, let $\Delta_n = \ln z_{\hat N_n(\varrho)}-\ln z_{\hat N_{n-1}(\varrho)}$. Set $t_n=\Delta_n+ \dots + \Delta_1=\ln z_{\hat N_n(\varrho)}-\ln z_{0}$.

\begin{prop}\label{mn-estm}
Under the assumptions of Proposition \ref{Lm1}, one can find for the modified system $P$ some positive numbers $\alpha, \beta, N_*>0$, depending on $f_1, f_2, \mathcal{E}$, so that for all $N\geq N_*$ and all long curves $\gamma$ we have that
$$
m_\gamma(\varrho=(x_0,z_0): z_n \geq e^{\alpha n}z_0, \text{ for all }n \geq N)\geq 1 - e^{-\beta N}.
$$
\end{prop}
\begin{proof}
Let $\mathcal{F}_k$ be the $\sigma$ algebra generated by the partition of $\gamma$ by $\hat N_k$ up to time $k$. Then by Lemma \ref{Lm1}
$$
\E_\gamma[e^{-\kappa (t_n-t_1)}]=\E_\gamma[\E_\gamma[e^{\kappa(t_n -t_1)}|\mathcal{F}_{n-1}]]=\E_\gamma[e^{\kappa(t_{n-1}-t_1)}\E[e^{\kappa \Delta_k}|\mathcal{F}_{n-1}]]
$$
$$
\leq \vartheta\E_\gamma[e^{-\kappa (t_{n-1}-t_1)}].
$$
Iterating
$$
\E_\gamma[e^{-\kappa (t_{n}-t_1)}]\leq \vartheta^n.
$$
Then for any $\lambda>0$
$$
e^\lambda m_\gamma(e^{-\kappa t_n} \geq  e^\lambda) \leq \vartheta^n |\gamma|.
$$
Take $\lambda=-\kappa \alpha n$. Then
$$
m_\gamma(e^{-\kappa t_n}\geq e^{-\kappa \alpha n}) \leq e^{\kappa \alpha n}\vartheta^n |\gamma|.
$$
Hence
$$
m_\gamma(t_n< \alpha n) \leq e^{\kappa \alpha n}\vartheta^n |\gamma|=e^{n(\kappa \alpha - |\ln \vartheta|)}|\gamma|.
$$
Obviously, if we take $\alpha$ sufficiently small then $a=-(\kappa \alpha - |\ln \vartheta|)>0$. Hence
$$
m_\gamma(\Delta_1 + \dots +\Delta_n\leq \alpha n) \leq e^{-an}|\gamma|.
$$

Equivalently
$$
m_\gamma(\varrho: z_{\hat N_n(\varrho)}\leq e^{n \alpha} z_0)\leq e^{-an}|\gamma|.
$$
Hence, there exists $c_0, N_0>0$ so that for all $N \geq N_0$
$$
m_\gamma(\varrho=(x_0, z_0): z_{\hat N_n(\varrho)}> e^{n \alpha} z_0\text{ for all }n \geq N) \geq (1-e^{-N c_0})|\gamma|.
$$
Thus, we have shown that there is energy acceleration at times $\{\hat N_n(\varrho)\}_{n}$ on a set of large measure. We now show that the acceleration will persist between the times $\hat N_n(\varrho)$ and $\hat N_{n+1}(\varrho)$ as well. If it won't and the energy between these times drops very low, then the interval $[\hat N_n(\varrho),\hat N_{n+1}(\varrho)]$ will have to be very long. But this can't happen too often due to the growth Lemma \ref{pr-delayed}.  Once energy acceleration is achieved for all times $n$, it will then remain to use the deviation bound in Lemma \ref{dev-prop} to transition from times $\hat N_n(\varrho)$ to $[an]$.

By Lemma \ref{dev-prop} we have that
$$
m_\gamma(\varrho:\hat N_n(\varrho)>an)\leq \vartheta_5^n |\gamma|.
$$
Hence, as above there exists $c_1, N_1>0$ so that for all $N \geq N_1$
$$
m_\gamma(\varrho:\hat N_n(z_0)>an: \text{ for all }n \geq N)\geq (1 - e^{-c_1 N})|\gamma|.
$$
Combining the estimates above, we can find constants $c_2, N_2$, so that for all $N \geq N_2$ there exists subsets $A_{N}\subset \gamma$ so that $m_\gamma(A_N)\geq (1-e^{-c_2 N})|\gamma|$ and for all $\varrho \in A_N$ and all $n \geq N$
$$
\hat N_n(\varrho)\leq  an, \quad \text{ and }  z_{\hat N_n(\varrho)}> e^{n \alpha} z_0.
$$
Let $\alpha_1<\alpha$ and assume that for some $\varrho \in A_N$ and $n \geq N$ there is an integer $1 \leq m(\varrho) \leq \hat N_{n+1}(\varrho) - \hat N_{n}(\varrho)$ so that
$$
z_{\hat N_n(\varrho) + m(\varrho)}<e^{n \alpha_1}z_0.
$$
Then, by Proposition \ref{M0} (3), for $d=\ell\left(\frac{1-f_2}{1-f_1}\right)$ we will have that
$$
d^m e^{n \alpha}z_0\leq d^m z_{\hat N_n}\leq z_{\hat N_n + m}\leq e^{n \alpha_1}z_0.
$$
Then
$$
d^m \leq e^{n(\alpha_1-\alpha)}.
$$
Hence
$$
m\geq \frac{n(\alpha_1-\alpha)}{\ln 1/d} = vn.
$$
Let $\gamma_n(\varrho)$ be the long curve that contains $\varrho$ at time $\hat N_n(\varrho)$. Then by Lemma \ref{pr-delayed} we have that $m_{\gamma_n}(\varrho: m(\varrho)\geq vn)\leq m_{\gamma_n}(\varrho: \hat N_{n+1}(\varrho)\geq vn)\leq b \vartheta_4^{[vn]-N_0}$. Thus, we will have that for some $c_3=c_3(\vartheta_4, N_0)$
$$
m_{\gamma_n}(\varrho: \exists m \text{ so that } z_{\hat N_n(\varrho) + m}<e^{n \alpha_1}z_0
)
\leq C\theta^{-N_0} (\theta^v)^{n}\leq e^{-c_3n}.
$$
Thus, there exists $c_4, N_4>0$ so that for all large $N \geq N_4$
$$
m_{\gamma}(\varrho: \text{ there exists } n\geq N \text{ and } m\geq 1, \text{ so that } z_{\hat N_n(\varrho) + m}<e^{n \alpha_1}z_0
)\le e^{-c_4N}.
$$
Adding this to the sets above, we see that there exist constants $c_5, N_5>0$ such that for all $N \geq N_5$ one can find a subset $B_N \subset A_N$ so that $|B_N|\geq 1-e^{-c_5N}$ and for all $\varrho \in B_N$ we have for all $n \geq N$ that
$$
\hat N_n(\varrho)<an, \quad   z_{\hat N_n(\varrho)}> e^{n \alpha} z_0,
$$
and for all $n \geq N$ and $\ell$, with $\hat N_n<\ell<\hat N_{n+1}$ we have
$$
z_\ell \geq e^{\alpha_1 n}z_0.
$$
Now consider the times $[an]$ for $n \geq N$. There exists $k$ so that $\hat N_k \leq [an]\leq \hat N_{k+1}$. Then for every $\varrho \in B_N$ and $n \geq N$
$$
z_{[an]}(\varrho)\geq e^{k \alpha_1} z_0(\varrho).
$$
But we have that $[an]>\hat N_n(\varrho)$, hence $k \geq n$. Thus
$$
z_{[an]}\geq e^{k \alpha_1} z_0 \geq e^{[an] \frac{\alpha_1}{a}} z_0.
$$
Between the times $[an]$ and $[a(n+1)]$, $z_k$ can decrease by at most $d^{a}$. Hence for all $[an]\leq s \leq [a(n+1)]$, and $\varrho \in B_N$ we will have that
$$
z_{s}\geq d^{a} e^{[an]\frac{\alpha_1}{a}} z_0\geq e^{s \frac{\alpha_2}{a}}z_0=e^{s \frac{\alpha_0}{a}}z_0.
$$
To get rid of $a$ in front of $N$ we will take $N$ larger. 
Thus we have shown that there exists $c_6,N_6>0$ so that for all $N\geq N_6$ there exists a set $B_N \subset \gamma$, with $m_\gamma(B_N)>1-e^{-c_6 N}$ so that for all $\varrho \in B_N$ we have for all $n \geq N$ that
$$
z_n \geq z_0 e^{(\alpha_0/a) n}.
$$
To finish the proof we set $\beta=c_6$, $\alpha=\alpha_0/a$ and $N_*=N$.
We remark that the last constants depend on the properties of the modified system $P_0$, which only depend on $f_1, f_2, \mathcal{E}$, if $\f$ and $\V$ are large enough.
\end{proof}

\medskip

Finally we show that the original system $P$ inherits a positive measure of exponentially escaping orbits from the modified system as they coincide above $V_0$.

\begin{proof}[Proof of main Theorem \ref{expAcc}]
We repeat the procedure of choosing the parameters $\f$ and $\V$ which has already been discussed above.
We first choose $\f$ so large that we have invariant unstable cones as in Proposition \ref{pcone} and the constants in the growth Lemma \ref{pr-delayed} are uniform for all $|\dot f_1|, |\dot f_2|>\f$. Note that, we first chose $\f$ and then $\V$, since $\V$ depends on the parameters $f_1, f_2,\dot f_1,\dot f_2$.
Note also, that the change of $\f$ is compensated by the decay of the size of long curves. 
Then, to assure the moment estimate in Proposition \ref{Lm1}, $N_0$ needs to be taken large. Its size will only depend on the constants in the growth lemma and $\mathcal{E}, f_1, f_2$. 
Since these constants are already fixed, we will only need to make the quantities $\f$ and $\V$ even larger. $\mathcal{E}$ contributes into the growth rate constant $\alpha$ and it measures how close $f_1$ and $f_2$ are.
Thus, for an appropriate choice of $\f$ and $\V$ we will have Proposition \ref{Lm1}. This, in turn, will imply Proposition \ref{mn-estm}.
It now remains to transition from the modified system $P_0$ to the original system $P$. Recall that by Proposition \ref{M0} the two systems coincide above $V_0$.

By Proposition \ref{mn-estm} there exist numbers $\alpha,\beta, N_*$ so that for all $N\geq N_*$ we have for every long curve $\gamma$ on the vertical line $\{\sigma=\sigma_0\}$ ($\sigma_0\in (0,2)$) that there is a set $A_N \subset \gamma$, with $m_\gamma(A_N)>(1-e^{-\beta N})|\gamma|$, so that for all $n \geq N$ and $\varrho =(x_0, z_0)\in A_N$ we have that 
$$
z_{n}(\varrho) \geq e^{\alpha n}z_0.
$$
Note that for all $n \geq N$ we have that $z_{n}(\varrho) \geq e^{\alpha N}z_0>z_0$. Hence, $z_n>z_0$ for all $n \geq N$. 
If we now take $z_0 \geq V_0 (1/d)^{N}$ then for all $n \leq N$ we will have that $z_{n} \geq V_0$. Thus, for all $\varrho \in A_N$ and $n\geq 1$ we will have that $z_n(\varrho)>V_0$. This means that the energy of the ball will never go below the threshold $V_0$ so our analysis in Proposition \ref{M0} will coincide with that of the original system $P$. If we now take $\iota$ so that $V_0 (1/d)^{N}=\V\iota^{N}$ and consider a foliation of $R_1 \cap [V_1, V_2]$ with long curves, then the result will follow. 
\end{proof}

\bibliography{ExponentialFermiAccelerationinSwitchingBilliard}

\begin{thebibliography}{10}

\bibitem{CM}
N.~Chernov and R.~Markarian.
\newblock {\em Chaotic Billiards}.
\newblock Providence, R.I. : American Mathematical Society, 2006.

\bibitem{CZ}
N.~Chernov and H.-K. Zhang.
\newblock On statistical properties of hyperbolic systems with singularities.
\newblock {\em Journal of Statistical Physics}, 136(4):615--642, 2009.

\bibitem{DS09}
J.~de~Simoi.
\newblock Stability and instability results in a model of {F}ermi acceleration.
\newblock {\em Discrete Contin. Dyn. Syst.}, 25(3):719--750, 2009.

\bibitem{deSD12}
J.~de~Simoi and D.~Dolgopyat.
\newblock Dynamics of some piecewise smooth {F}ermi--{U}lam models.
\newblock {\em Chaos: An Interdisciplinary Journal of Nonlinear Science},
  22(2):026124, 2012.

\bibitem{Dol08potential}
D.~Dolgopyat.
\newblock Bouncing balls in non-linear potentials.
\newblock {\em Discrete Contin. Dyn. Syst.}, 22(1-2):165--182, 2008.

\bibitem{Dol08}
D.~Dolgopyat.
\newblock Fermi acceleration.
\newblock In {\em Geometric and probabilistic structures in dynamics}, volume
  469 of {\em Contemp. Math.}, pages 149--166. Amer. Math. Soc., Providence,
  RI, 2008.

\bibitem{Fermi49}
E.~Fermi.
\newblock On the origin of the cosmic radiation.
\newblock {\em Phys. Rev.}, 75:1169--1174, Apr 1949.

\bibitem{GRKT12}
V.~Gelfreich, V.~Rom-Kedar, and D.~Turaev.
\newblock Fermi acceleration and adiabatic invariants for non-autonomous
  billiards.
\newblock {\em Chaos}, 22(3):033116, 21, 2012.

\bibitem{GRKT14}
V.~Gelfreich, V.~Rom-Kedar, and D.~Turaev.
\newblock Oscillating mushrooms: adiabatic theory for a non-ergodic system.
\newblock {\em J. Phys. A}, 47(39):395101, 21, 2014.

\bibitem{GT08nonauto}
V.~Gelfreich and D.~Turaev.
\newblock Fermi acceleration in non-autonomous billiards.
\newblock {\em Journal of Physics A: Mathematical and Theoretical},
  41(21):212003, may 2008.

\bibitem{GT08ham}
V.~Gelfreich and D.~Turaev.
\newblock Unbounded energy growth in hamiltonian systems with a slowly varying
  parameter.
\newblock {\em Commun. Math. Phys.}, 2008.

\bibitem{WH1963}
W.~Hoeffding.
\newblock Probability inequalities for sums of bounded random variables.
\newblock {\em J. of American Statistical Association}, 58(301):13--30, 1963.

\bibitem{KMKC95}
J.~Koiller, R.~Markarian, S.O. Kamphorst, and S.P. de~Carvalho.
\newblock Time-dependent billiards.
\newblock {\em Nonlinearity}, 8(6):983--1003, nov 1995.

\bibitem{kuOr20}
M.~Kunze and R.~Ortega.
\newblock Escaping orbits are rare in the quasi-periodic fermi–ulam
  ping-pong.
\newblock {\em Ergodic Theory and Dynamical Systems}, 40(4):975–991, 2020.

\bibitem{LaLe91}
S.~Laederich and M.~Levi.
\newblock Invariant curves and time-dependent potentials.
\newblock {\em Ergodic Theory and Dynamical Systems}, 11(2):365--378, 1991.

\bibitem{LKSK04}
E.D. Leonel, J.~Kamphorst, L.~da~Silva, and S.O. Kamphorst.
\newblock On the dynamical properties of a fermi accelerator model.
\newblock {\em Physica A: Statistical Mechanics and its Applications},
  331(3):435 -- 447, 2004.

\bibitem{LLC80}
A.~J. Lichtenberg, M.~A. Lieberman, and R.~H. Cohen.
\newblock Fermi acceleration revisited.
\newblock {\em Phys. D}, 1(3):291--305, 1980.

\bibitem{Or02}
R.~Ortega.
\newblock Dynamics of a forced oscillator having an obstacle.
\newblock In {\em Variational and topological methods in the study of nonlinear
  phenomena ({P}isa, 2000)}, volume~49 of {\em Progr. Nonlinear Differential
  Equations Appl.}, pages 75--87. Birkh\"{a}user Boston, Boston, MA, 2002.

\bibitem{Pu77}
L.~D. Pustylnikov.
\newblock Stable and oscillating motions in nonautonomous dynamical systems.
  {II}.
\newblock {\em Trudy Moskov. Mat. Ob\v{s}\v{c}.}, 34:3--103, 1977.

\bibitem{Pu83}
L.~D. Pustylnikov.
\newblock On {U}lam's problem.
\newblock {\em Theoret. and Math. Phys.}, 57:1035--1038, 1983.

\bibitem{Pu94}
L.~D. Pustylnikov.
\newblock Existence of invariant curves for maps close to degenerate maps, and
  a~solution of the~{F}ermi--{U}lam problem.
\newblock {\em Mat. Sb.}, 185:113--124, 1994.

\bibitem{STR}
K.~Shah, D.~Turaev, and V.~Rom-Kedar.
\newblock Exponential energy growth in a {F}ermi accelerator.
\newblock {\em Phys. Rev. E}, 81:056205, May 2010.

\bibitem{Ulam61}
S.M. Ulam.
\newblock On some statistical properties of dynamical systems.
\newblock In {\em Proceedings of the Fourth Berkeley Symposium on Mathematical
  Statistics and Probability, Volume 3: Contributions to Astronomy,
  Meteorology, and Physics}, pages 315--320, Berkeley, Calif., 1961. University
  of California Press.

\bibitem{Viana97}
M.~Viana.
\newblock {\em Stochastic Dynamics of Deterministic Systems}.
\newblock IMPA, 1997.

\bibitem{Zhar98}
V.~Zharnitsky.
\newblock Instability in {F}ermi-{U}lam ping-pong problem.
\newblock {\em Nonlinearity}, 11(6):1481--1487, nov 1998.

\bibitem{Zhar2000}
V.~Zharnitsky.
\newblock Invariant curve theorem for quasiperiodic twist mappings and
  stability of motion in the {F}ermi-{U}lam problem.
\newblock {\em Nonlinearity}, 13(4):1123--1136, 2000.

\bibitem{Zhou20}
J.~Zhou.
\newblock A rectangular billiard with moving slits.
\newblock {\em Nonlinearity}, 33(4):1542--1571, Feb 2020.

\bibitem{Zhou21}
J.~Zhou.
\newblock A piecewise smooth fermi–ulam pingpong with potential.
\newblock {\em Ergodic Theory and Dynamical Systems}, page 1–24, 2021.

\end{thebibliography}
\bibliographystyle{plain}

\end{document}